\documentclass[reqno]{amsart}
\title{Beyond G\"ollnitz' Theorem II: arbitrarily many primary colors}
\author{Isaac KONAN} \address{IRIF \\ University
of Paris Diderot \\  Paris,  75013, France}
\email{konan@irif.fr}
\date{}
\usepackage{amssymb}
\usepackage{amsmath}
\usepackage{amsthm}
\usepackage[top=1in,bottom=1in,left=1in,right=1in,includehead]{geometry}
\usepackage{enumitem}
\usepackage{xcolor}
\usepackage{graphicx}
\usepackage{pgfarrows,pgfnodes}
\usepackage{url}
\usepackage{textcomp}
\usepackage{ytableau}
\usepackage{dsfont}
\usepackage{tikz}
\usepackage[toc,page]{appendix}

\allowdisplaybreaks
\newcommand{\m}{\medbreak}
\newcommand{\bi}{\bigbreak}

\definecolor{foge}{rgb}{0.1, 0.6, 0.1}

\newcommand{\So}{\textbf{Step 1 }}
\newcommand{\N}{\mathbb{N}}
\newcommand{\Soo}{\textbf{Step 1}}
\newcommand{\St}{\textbf{Step 2 }}
\newcommand{\Stt}{\textbf{Step 2}}
\newcommand{\Thm}[1]{Theorem \ref{#1}}

\newcommand{\Lem}[1]{Lemma \ref{#1}}
\newcommand{\Prp}[1]{Proposition \ref{#1}}
\newcommand{\Sct}[1]{Section \ref{#1}}
\newcommand{\Expl}[1]{Example \ref{#1}}
\newcommand{\Od}{\mathcal{O}}
\newcommand{\C}{\mathcal{C}}

\newcommand{\Br}{\mathbf{Br}_\nu}
\newcommand{\ts}{\mathcal{TS}}
\newcommand{\Cc}{\mathcal{C}_\rtimes}
\newcommand{\Ccc}{\mathcal{SP}_\rtimes}

\newcommand{\Pp}{\mathcal{P}}
\newcommand{\Sc}{\mathcal{S}}

\newcommand{\E}{\mathcal{E}}
\newcommand{\Ee}{\mathcal{E}_1}
\newcommand{\Eee}{\mathcal{E}_2}
\newcommand{\nt}{\displaystyle{\not\triangleright\,\,\,}}

\newcommand{\la}{\lambda}
\newcommand{\sss}{\{1,\ldots,s\}}
\newcommand{\ssss}{\{1,\ldots,s-1\}}
\newcommand{\pat}{\mathbf{pattern}}

\numberwithin{equation}{section}
\newtheorem{theo}{Theorem}[section]
\newtheorem{prop}[theo]{Proposition}
\newtheorem{lem}[theo]{Lemma}
\newtheorem{cor}[theo]{Corollary}
\newtheorem{rem}[theo]{Remark}
\newtheorem{ex}[theo]{Example}
\newtheorem{exs}[theo]{Examples}
\theoremstyle{definition} \newtheorem{deff}[theo]{Definition}

\begin{document}
\maketitle
\begin{abstract}
In $2003$, Alladi, Andrews and Berkovich proved a four
parameter partition identity lying beyond a celebrated identity of
G\"ollnitz.   Since then it has been an open problem to extend their
work to five or more parameters.   In part I of this pair of papers, we took a first step in this direction by giving a bijective proof of a reformulation of their result.  We
introduced forbidden patterns, bijectively proved a ten-colored partition identity,  and then related, by another bijection, our identity to the Alladi-Andrews-Berkovich identity. 
\\In this second paper, we state and bijectively prove an $\frac{n(n+1)}{2}$-colored partition identity beyond G\"ollnitz' theorem for any number $n$ of primary colors, along with the full set of the $\frac{n(n-1)}{2}$ secondary colors as the product of two distinct primary colors, generalizing the identity proved in the first paper.  Like the ten-colored partitions, our family of $\frac{n(n+1)}{2}$-colored partitions satisfy some simple minimal difference conditions while avoiding forbidden patterns. Furthermore, the $\frac{n(n+1)}{2}$-colored partitions have some remarkable properties, as they can be uniquely represented by oriented rooted forests which record the steps of the bijection.
\end{abstract}
\section{Introduction and Statements of Results}
\subsection{History}
A partition of a positive integer $n$ is a non-increasing sequence of positive integers whose sum is equal to $n$. For example, the partitions of $7$ are 
\[(7),(6,1),(5,2),(5,1,1),(4,3),(4,2,1),(4,1,1,1),(3,3,1),(3,2,2),(3,2,1,1),\]
\[(3,1,1,1,1),(2,2,2,1),(2,2,1,1,1),(2,1,1,1,1,1)\,\,\text{and}\,\,(1,1,1,1,1,1,1)\,\cdot\] 
The study of partition identities has been a center of interest for centuries, dating back to
Euler's proof of the following identity, 
\begin{equation}
(-q;q)_\infty = \frac{1}{(q;q^2)_\infty}\,,
\end{equation}
where \[(x;q)_m = \prod_{k=0}^{m-1} (1-xq^k)\,\,,\]
for  any $m\in \mathbb{N}\cup \{\infty\}$ and $a,q$ such that $|q|<1$. From a combinatorial viewpoint, this can be 
formulated by the following statement: there are as many partitions of $n$ into distinct parts as partitions of $n$ into odd parts.
\bi 
The study of integer partitions underwent a significant advancement with the works of Rogers-Ramanujan at the beginning of the past century. Following in their tracks, Schur found another simple identity \cite{Sc26}, stating that the number of partitions of $n$ into distinct parts congruent to $\pm 1 \mod 3$ is equal to the number of partitions of $n$ where parts differ by at least three and multiples of three differ by at least six. In the spirit of Schur's identity, G\"ollnitz proved in \cite{GO67} that the number of partitions of $n$ into distinct parts  congruent to $2,4,5 \mod 6$ is equal to the number of partitions of $n$ into parts  different from $1$ and $3$, and where parts differ by at least six with equality only if parts are congruent to $2,4,5 \mod 6$. Looking at the set of partitions involved in the previous identities, one may view G\"ollnitz' identity as embedded in Schur's identity. It also works in the opposite way, as we now describe. 
\bi In seminal work of Alladi-Gordon \cite{AG93}, they pointed out a refinement of Schur's identity,
where they introduced the weighted words method with the use of two primary colors $a,b$ along with a secondary color $ab$. Later, Alladi-Andrews-Gordon \cite{AAG95} found a refinement of
G\"ollnitz' identity, with the use of weighted words with three primary colors $a,b,c$ and three secondary colors $ab,ac,bc$, which indeed implies the refinement of Schur's identity. Further explanation of these two refinements is given in the first part of this series \cite{IK}.
\bi
It was an open problem to find a partition identity beyond G\"ollnitz' theorem arising from four primary colors. In \cite{AAB03}, Alladi, Andrews, and Berkovich solved this problem. Their result  uses four primary colors, the full set of
secondary colors, along with one quaternary color $abcd$, and can be described as follows.
We consider parts that occur in eleven colors $\{a,b,c,d,ab,ab,ad,bc,bd,cd,abcd\}$ and ordered as follows:
\begin{equation}\label{quat}
1_{abcd}<1_{ab}<1_{ac}<1_{ad}<1_a<1_{bc}<1_{bd}<1_b<1_{cd}<1_c<1_d<2_{abcd}<\cdots\,\cdot
\end{equation}
Let us consider the partitions with the size of the secondary parts greater than one and satisfying the minimal difference conditions in 
\begin{equation}\label{tab2}
\begin{array}{|c|cccc|ccc|cc|c|}
\hline
_{\lambda_i}\setminus^{\lambda_{i+1}}&ab&ac&ad&a&bc&bd&b&cd&c&d\\
\hline
ab&2&2&2&2&2&2&2&2&2&2\\
ac&1&2&2&2&2&2&2&2&2&2\\
ad&1&1&2&2&2&2&2&2&2&2\\
a&1&1&1&1&2&2&2&2&2&2\\
\hline
bc&1&1&1&1&2&2&2&2&2&2\\
bd&1&1&1&1&1&2&2&2&2&2\\
b&1&1&1&1&1&1&1&2&2&2\\
\hline
cd&1&1&1&1&1&1&1&2&2&2\\
c&1&1&1&1&1&1&1&1&1&2\\
\hline
d&1&1&1&1&1&1&1&1&1&1\\
\hline
\end{array}\,,
\end{equation}
and such that parts with color $abcd$ differ by at least $4$, and the smallest part with color $abcd$ is at least equal to  $4+2\tau-\chi(1_a\text{ is a part})$, where $\tau$ is the number of primary and secondary parts in the partition. Here, $\chi(A)$ equals $1$ if the proposition $A$ is true and $0$ if not, and the term \textit{minimal difference conditions} means that, for $\lambda = (\lambda_1,\ldots,\lambda_s)$ with the parts $\la_k$ colored by $c(\la_k)$, we have for all $i\in \ssss$ that the value $\lambda_i-\lambda_{i+1}$ is at least equal to the value corresponding to the row $c(\la_i)$ and the column $c(\la_{i+1})$.
Their theorem is then stated as follows.
\begin{theo}[Alladi-Andrews-Berkovich]\label{th2}
Let $u,v,w,t,n$ be non-negative integers. Denote by $A(u,v,w,t,n)$ the number of partitions of $n$ into $u$ distinct parts with color $a$, $v$ distinct parts with color $b$, $w$ distinct parts with color $c$ and $t$ distinct parts with color $d$, and denote by $B(u,v,w,t,n)$ the number of partitions of $n$ satisfying the conditions above, with $u$ parts with color $a,ab,ac,ad$ or $abcd$, $v$ parts with color $b,ab,bc,bd$ or $abcd$, $w$ parts with color $c,ac,bc,cd$ or $abcd$ and $t$ parts with color $d,ad,bd,cd$ or $abcd$.
We then have $A(u,v,w,t,n)=B(u,v,w,t,n)$ and the identity
\begin{equation}
\sum_{u,v,w,t,n\geq 0} B(u,v,w,t,n)a^ub^vc^wd^tq^n =
(-aq;q)_\infty (-bq;q)_\infty(-cq;q)_\infty(-dq;q)_\infty\,\cdot
\end{equation}
\end{theo}
Note that when $d=0$, we recover Alladi-Andrews-Gordon's refinement of G\"ollnitz' identity (see \cite{IK} for more details). Their main tool was an intricate $q$-series identity.
\bi
In part I of this series \cite{IK}, we showed an equivalent version of \Thm{th2}. In fact,  
we supposed that the parts occur in only primary colors $a,b,c,d$ and secondary colors $ab,ac,ad,bc,bd,cd$, and are ordered as in \eqref{quat} by omitting quaternary parts:
\begin{equation}\label{cons}
1_{ab}< 1_{ac}< 1_{ad} < 1_{a}< 1_{bc} < 1_{bd}< 1_b< 1_{cd}< 1_c<1_d< 2_{ab}<\cdots\,\cdot
\end{equation}
We then considered the partitions with the size of the secondary parts greater than one and satisfying the minimal difference conditions in
\begin{equation}\label{tab1}
\begin{array}{|c|cccc|ccc|cc|c|}
\hline
_{\lambda_i}\setminus^{\lambda_{i+1}}&ab&ac&ad&a&bc&bd&b&cd&c&d\\
\hline
ab&2&2&2&2&2&2&2&2&2&2\\
ac&1&2&2&2&2&2&2&2&2&2\\
ad&1&1&2&2&\underline{1}&2&2&2&2&2\\
a&1&1&1&1&2&2&2&2&2&2\\
\hline
bc&1&1&1&1&2&2&2&2&2&2\\
bd&1&1&1&1&1&2&2&2&2&2\\
b&1&1&1&1&1&1&1&2&2&2\\
\hline
cd&\underline{0}&1&1&1&1&1&1&2&2&2\\
c&1&1&1&1&1&1&1&1&1&2\\
\hline
d&1&1&1&1&1&1&1&1&1&1\\
\hline
\end{array}\,,
\end{equation} 
and which avoid the forbidden patterns 
\begin{equation}
((k+2)_{cd},(k+2)_{ab},k_c),((k+2)_{cd},(k+2)_{ab},k_d), ((k+2)_{ad},(k+1)_{bc},k_{a})\,,
\end{equation}
except the pattern $(3_{ad},2_{bc},1_a)$ which is allowed, and we obtained the following theorem:
\begin{theo}\label{th1}
Let $u,v,w,t,n$ be non-negative integers. Denote by $A(u,v,w,t,n)$ the number of partitions of $n$ into $u$ distinct parts with color $a$, $v$ distinct parts with color $b$, $w$ distinct parts with color $c$ and $t$ distinct parts with color $d$, and denote by $B(u,v,w,t,n)$ the number of partitions of $n$ satisfying the conditions above, with $u$ parts with color $a,ab,ac$ or $ad$, $v$ parts with color $b,ab,bc$ or $bd$, $w$ parts with color $c,ac,bc$ or $cd$ and $t$ parts with color $d,ad,bd$ or $cd$.
We then have $A(u,v,w,t,n)=B(u,v,w,t,n)$, and the corresponding $q$-series identity is given by
\begin{equation}\label{series}
\sum_{u,v,w,t,n\in \N} B(u,v,w,t,n)a^ub^vc^wd^tq^n =(-aq;q)_\infty (-bq;q)_\infty(-cq;q)_\infty(-dq;q)_\infty\,\cdot
\end{equation}
\end{theo}
The proof of \Thm{th1} consisted of a bijection established between the two sets of partitions. We also used a second bijection to show that \Thm{th1} is equivalent to \Thm{th2}. 
\bi 
By specializing the variables in \Thm{th1}, one can deduce
many partition identities.   For example, by considering the following transformation in \eqref{series}
\begin{equation}\label{dila}
\left\lbrace 
\begin{array}{l r c l }
\text{dilation :} &q &\mapsto&q^{12}\\
\text{translations :} &a,b,c,d &\mapsto&q^{-8},q^{-4},q^{-2},q^{-1}\\
\end{array}
\right. \,,
\end{equation}
we obtain a corollary of \Thm{th1}.
\begin{cor}
For any positive integer $n$, the number of partitions of $n$ into distinct parts congruent to $-2^3,-2^2,-2^1,-2^0\mod 12$ is equal to the number of partitions of $n$ into parts not congruent to $1,5\mod 12$ and different from $2,3,6,7,9$, such that the difference between two consecutive parts is greater than $12$ up to the following exceptions:
\begin{itemize}
\item $\la_i-\la_{i+1}= 9\Longrightarrow \la_i\equiv \pm 3 \mod 12$ and $\la_i-\la_{i+2}\geq 24$,
\item $\la_i-\la_{i+1}= 12\Longrightarrow\la_i\equiv -2^3,-2^2,-2^1,-2^0\mod 12$,
\end{itemize}
except that the pattern $(27,18,4)$ is allowed. 
\end{cor}
\begin{ex}
For example, with $n=49$, the partitions of the first kind are 
\[(35,10,4),(34,11,4),(28,11,10),(23,22,4),\]
\[(23,16,10),(22,16,11)\,\,\text{and}\,\, (16,11,10,8,4)\]
and the partitions of the second kind are
\[(35,14),(34,15),(33,16),(45,4),(39,10),(38,11)\,\,\text{and}\,\,(27,18,4)\,\cdot\]
\end{ex}
\bi
The main goal of this paper is to give a general result beyond G\"ollnitz' theorem, by proving an analogue of \Thm{th1} for an arbitrary finite set of primary colors.
\subsection{Statement of Results}
Let $\C = \{a_1,\ldots,a_n\}$ be an ordered set of primary colors, with 
$ a_1<\cdots<a_n$
 and let $\Cc = \{a_ia_j:  1\leq i<j\leq n\}$ be the set of secondary colors.
 \bi
 We can naturally extend the order from $\C$ to $\C\sqcup \Cc$ with
\begin{align}\label{orD}
a_1a_2<\cdots<a_1a_n<a_1<a_2a_3<\cdots<a_2a_n<a_2<\cdots<a_{i-1}\\
<a_ia_{i+1}<\cdots<a_ia_n<a_i<\cdots<a_{n-1}a_n<a_{n-1}<a_n\,\cdot\nonumber
\end{align}
We also set 
\begin{equation}\label{special}
\Ccc = \{(a_ka_l,a_ia_j)\in \Cc^2: \, i<j<k<l\text{ or } k<i<j<l\} 
\end{equation}
 to be the set of the \textit{special pairs} of secondary colors. Note that the pairs of $\Ccc$ use four different primary colors.
Let us now define the lexicographic order $\succ$ on the set of colored parts by the following relation:
\begin{equation}\label{lex}
k_p\succ l_{q} \Longleftrightarrow k-l\geq \chi(p\leq q)\,\cdot
\end{equation}
Explicitly, this gives the order 
\begin{equation}\label{cons1}
1_{a_1a_2}\prec\cdots\prec 1_{a_n}\prec 2_{a_1a_2} \prec\cdots\prec 2_{a_n}\prec 3_{a_1a_2}\prec \cdots \,\cdot
\end{equation}
\begin{deff}
Let $\Pp$ be the set of the parts with primary color, and let $\Sc$ be the set of the parts with secondary color and size greater than one.
We then  define two relations  $\triangleright$ and $\gg$  on $\Pp\sqcup \Sc$ as follows : 
\begin{equation}\label{Ord}
k_p\triangleright l_{q}\Longleftrightarrow\left\lbrace
\begin{array}{ll}
k_p\succeq (l+1)_{q}&\text{if}\quad p\,\,\text{or}\,\,q\in \C\\
k_p\succ (l+1)_{q}&\text{if}\quad p\,\,\text{and}\,\,q\in \Cc
\end{array}\right.\,,
\end{equation}
and 
\begin{equation}\label{Ordd}
k_p\gg l_{q}\Longleftrightarrow\left\lbrace
\begin{array}{ll}
k_p\succeq (l+1)_{q}&\text{if}\quad p\,\,\text{or}\,\,q \in \C\\
k_p\succ (l+1)_{q}&\text{if}\quad (p,q) \in \Cc^2\setminus \Ccc\\
k_p\succ l_{q}&\text{if}\quad (p,q)\in\Ccc
\end{array}\right.\,\cdot
\end{equation}
\end{deff}
Note that $k_p\triangleright l_{q}$ implies $k_p\gg l_{q}$.
We can easily check that in the case $n=4$ and $\C =\{a<b<c<d\}$, the relations $\triangleright$ and $\gg$ establish some minimal differences $k-l$ that correspond respectively to the minimal differences 
$\la_i-\la_{i+1}$ in 
\eqref{tab2} and \eqref{tab1}. We also remark that these differences constitute an exhaustive list of all the minimal differences for our relations, since at most four primary colors occur in any pair of colors in $\C\sqcup\Cc$.
\begin{deff}\label{defff}
A secondary color is just a product of two primary colors. For any type of partition $\la$, its size $|\la|$ is the sum of its part sizes.
\begin{enumerate}
\item We denote by $\Od$ the set of partitions with parts in $\Pp$ and well-ordered by $\succ$. We then have that 
$\la \in \Od$ if and only if there exist $\la_1\succ\cdots\succ \la_t \in \Pp$ such that $\la=(\la_1,\ldots,\la_t)$. We set 
$c(\la_i)$ to be the color of $\la_i$ in $\C$, and $C(\la) = c(\la_1)\cdots c(\la_t)$ as a commutative product of colors in $<\C>$.
\item We denote by $\E$ the set of partitions with parts in $\Pp\sqcup \Sc$ and well-ordered by $\gg$. We then have that 
$\nu \in \E$ if and only if there exist $\nu_1\gg\cdots\gg\nu_t \in \Pp\sqcup \Sc$ such that $\nu=(\nu_1,\ldots,\nu_t)$. We set colors $c(\nu_i)\in \C\sqcup\Cc$ depending on whether $\nu_i$ is in  $\Pp$ or $\Sc$, and we also define 
$C(\nu)=c(\nu_1)\cdots c(\nu_t)$ \textbf{seen} as a commutative product of colors in $\C$.
\item We finally  denote by $\Eee$ the subset of partitions of $\E$ with parts well-ordered by $\triangleright$.
\end{enumerate} 
\end{deff}
We can now state the first result of our paper.
\begin{theo}\label{th3}
Let $m$ be a non-negative integer and $C$ a commutative product of primary colors in $\C$. Denote by $U(C,m)$ the number of partitions $\la$ in $\Od$  with $(C(\la),|\la|) =(C,m)$, and 
denote by $V(C,m)$  the number of partitions $\nu$ in $\E$ with $(C(\nu),|\nu|) =(C,m)$. We then have the following inequality :
\begin{equation}
U(C,m)\leq V(C,m)\,\cdot
\end{equation}
\end{theo}
The previous theorem implies that $\Od$ can be associated to a set $\Ee$ such that $\Ee\subset \E$. We define this set $\Ee$ using two technical tools : the \textbf{different-distance} and the \textbf{bridge}. 
The definition of the different-distance is stated here, while the definition of the bridge, more intricate, will be given in \Sct{sct5}.
\begin{deff}
Let $\lambda=(\lambda_1,\cdots,\lambda_s)$ be a \textit{sequence}, where the elements $\lambda_i$ belong to a set of colored numbers ordered by a relation $\succeq$, and let $d$ be a positive number. For any $i,j \in \sss$, we say that $\lambda_i$ \textit{is $d$-different-distant from} $\lambda_j$ if we have the following relation:
\begin{equation}\label{eq:difdist}
\lambda_i\succeq \lambda_j + d(j-i)\,\cdot
\end{equation}
Note that the relation "being $d$-different-distant from" is transitive, as $\lambda_i$ is $d$-different-distant from 
$\lambda_j$ and $\lambda_j$ is $d$-different-distant from $\lambda_k$ implies that $\lambda_i$ is $d$-different-distant from 
$\lambda_k$.
\end{deff}
A good example of a partition having such a property is a partition $\nu=(\nu_1,\cdots,\nu_s)\in \Eee$. In fact, by \eqref{Ord}, we recursively obtain for any  $i\leq j \in \sss$ that $\nu_i$ is $1$-different-distant from $\nu_j$. This is not true in general when $\nu\in \E$, as by \eqref{Ord} and \eqref{Ordd}, a part $\nu_i$ not well-ordered with $\nu_{i+1}$ in terms of $\triangleright$  is also not $1$-different-distant from $\nu_{i+1}$.
\bi
The main theorem of this paper and generalization of \Thm{th1} can be stated as follows. 
\begin{theo}\label{th4}
Let $\Ee$ be the set of partitions $\nu = (\nu_1,\ldots,\nu_s)\in \E$ such that, for all $i\in \ssss$ with 
\begin{equation}\label{real}
\nu_{i-1}\triangleright \nu_i  \not\triangleright\,\nu_{i+1}\,,
\end{equation}
the part $\nu_i$ is $1$-different-distant from its \textbf{bridge}. Then, for any non-negative integer $m$ and any commutative product of primary colors $C$ in $\C$, by setting $U(C,m)$ as before in \Thm{th3}, and 
by setting $W(C,m)$ to be the number of partitions $\nu$ in  $\Ee$  with $(C(\nu),|\nu|) =(C,m)$, we then have that $U(C,m)=W(C,m)$ and the identity
\begin{equation}\label{eq:final}
 \sum_{m,u_1,\ldots,u_n\geq 0} W(\prod_{i=1}^n a_i^{u_i},m)\prod_{i=1}^n a_i^{u_i}q^m = \sum_{m,u_1,\ldots,u_n\geq 0} U(\prod_{i=1}^n a_i^{u_i},m)\prod_{i=1}^n a_i^{u_i} q^m = (-a_1q;q)_\infty\cdots (-a_nq;q)_\infty \,\cdot
\end{equation}
\end{theo} 
This may be compared with another result of the author.  
In his generalization of Siladi\'c's partition theorem \cite{IK19}, he used the same set of $n$ primary colors, along with the total set of the $n^2$ non-commutative secondary colors $a_ia_j$ for $i,j\in\{1,\ldots,n\}$, and gave an identity with the same product as \eqref{eq:final}. Another identity, discovered by Corteel and Lovejoy \cite{CL06}, relates the same set of partitions, with primary colored parts, to a set of partitions with parts having some colors as products of at most $n$ different primary colors, giving $2^n-1$ colors in total.
\bi Note that by definition, 
a partition in $\Eee$ never satisties \eqref{real}, so that the definition of $\Ee$ still holds for this partition. We thus have $\Eee\subset \Ee \subset \E$.
We also remark that $\Ccc$ is empty for $\C$ with fewer than four primary colors, so that in that case, $\Eee=\E$. Therefore, \Thm{th4} implies the Alladi-Andrews-Gordon refinement of G\"ollnitz' identity.
For $n\geq 4$, the set $\Ee$ can be seen as a subset of $\E$ that avoids some patterns. When $n=4$, we show that the forbidden patterns are the ones described in \Thm{th1}. For $n>4$, the enumeration of forbidden patterns becomes more intricate.   
\bi
The paper is organized as follows. In \Sct{sct2}, we will present some tools that will be useful for the proof of \Thm{th3} and \Thm{th4}.
After that, in \Sct{sct3}, we will give two mappings $\Phi$ and $\Psi$ for \Thm{th3} that preserve the size and the color product of partitions. Then, in \Sct{sct4}, we will prove \Thm{th3} by showing that 
$\Phi(\Od)\subset \E$ and $\Psi\circ\Phi_{|\Od} = Id_{|\Od}$.
In \Sct{sct5}, we will set $\Ee=\Phi(\Od)$, describe the notion of bridge, and prove \Thm{th4}.
In \Sct{sct6}, we explain how to generate the forbidden patterns of \Thm{th4}, and we especially retrieve in the case of four primary colors the three forbidden patterns as enumerated in \Thm{th1}, and we prove  that, for more than four primary colors, there is an infinite set of forbidden patterns.
Finally, in \Sct{sct7}, we relate the mapping $\Psi$ to \textit{Motzkin paths} and \textit{oriented rooted forests}, giving new perspectives for the study of the forbidden patterns.
\\\\We postpone the proofs of the technical lemmas and propositions respectively to Sections $8$ and $9$.
\section{Preliminaries}\label{sct2}
\subsection{The setup}
Let us first analyze the secondary parts in $\Sc$. For any $1\leq i<j\leq n$, and any positive integer $k$,  we have
\begin{align}
\label{half1}(2k)_{a_ia_j} &= k_{a_j}+k_{a_i}\\
(2k+1)_{a_ia_j} &= (k+1)_{a_i}+k_{a_j}\,\cdot \nonumber
\end{align}
In fact,  any secondary part in $\Sc$ with color $a_ia_j$ can be uniquely written as the sum of two consecutive parts in $\Pp$ with colors $a_i$ and $a_j$ in terms of $\succ$. 
\begin{deff}
For any $1\leq i<j\leq n$, we define the functions $\alpha$ and $\beta$ on $\Sc$ by 
\begin{equation}\label{ab}
\alpha: \left\lbrace \begin{array}{l c l}
2k_{a_ia_j}&\mapsto& k_{a_j}\\
(2k+1)_{a_ia_j}&\mapsto&(k+1)_{a_i}
\end{array}\right.\qquad \text{and}\qquad\beta: \left\lbrace \begin{array}{l c l}
2k_{a_ia_j}&\mapsto& k_{a_i}\\
(2k+1)_{a_ia_j}&\mapsto&k_{a_j}
\end{array}\right.\,,
\end{equation}
respectively named \textit{upper} and \textit{lower halves}.
\end{deff}
 One can check that for any $k_{a_ia_j}\in \Sc$, 
\begin{equation}\label{abba}
\alpha((k+1)_{a_ia_j}) = \beta(k_{a_ia_j})+1\quad\text{and}\quad \beta((k+1)_{a_ia_j})=\alpha(k_{a_ia_j})\,\cdot
\end{equation}
In the previous sum, adding an integer to a part only changes its size but does not change its color.
We can then deduce by induction that for any $m\geq 0$,
\begin{equation}\label{aj}
\alpha((k+m)_{a_ia_j})\preceq \alpha(k_{a_ia_j})+m \quad\text{and}\quad \beta((k+m)_{a_ia_j})\preceq \beta(k_{a_ia_j})+m\,\cdot
\end{equation}
\begin{rem}\label{moitmoit}
In fact, we have 
\begin{equation}
\alpha((k+2m)_{a_ia_j})=\alpha(k_{a_ia_j})+m \quad\text{and}\quad \beta((k+2m)_{a_ia_j})= \beta(k_{a_ia_j})+m\,\cdot
\end{equation}
\end{rem}
\bi
\begin{rem}
 Let us consider a partition $\la$ in $\Od$. By definition \eqref{defff}, 
it does not belong to $\E$ if and only if it has two consecutive parts $\la_i,\la_{i+1}$ such that $\la_i\not\gg \la_{i+1}$. We then have by \eqref{Ordd} that
\begin{equation}\label{nog}
\la_i\succ\la_{i+1}\quad\text{and}\quad \la_i\not\gg \,\la_{i+1}\quad\Longleftrightarrow \quad\la_{i+1}+1\succ \la_i\succ\la_{i+1}\,\cdot
\end{equation}
An equivalent reformulation consists in saying that $\la_i$ and $\la_{i+1}$ are two primary parts with distinct colors, consecutive in terms of $\succ$. Then, by \eqref{ab},
$\la_i+\la_{i+1}$ can be seen as the unique secondary part with respectively $\la_i$ and $\la_{i+1}$ as its upper and  lower halves.
\end{rem}
\bi
\subsection{Technical lemmas}
We will state some important lemmas for the proof of \Thm{th3} and \Thm{th4}. The proofs can be found in Section $8$.
\bi
\begin{lem}[\textbf{Ordering primary and secondary parts}]\label{lem1}
For any $(l_p,k_{q})\in\Pp\times \Sc$, we have the following equivalences:
\begin{align}
&\quad\l_p\not \gg k_{q}\Longleftrightarrow (k+1)_{q}\gg (l-1)_p\label{oe}\,,\\
&\quad l_p \gg \alpha(k_{q})\Longleftrightarrow \beta((k+1)_{q})\not \succ (l-1)_p\label{eo}\,\cdot
\end{align}
\bi
\end{lem}
\begin{lem}[\textbf{Ordering secondary parts}]\label{lem2}
Let us consider the table $\Delta$ in \eqref{tab1}. Then, for any secondary colors $p,q\in \Cc$, 
\begin{equation}\label{gam}
\Delta(p,q) = \min\{k-l:\beta(k_{p})\succ \alpha(l_{q})\}\,\cdot
\end{equation} 
Moreover, if the secondary parts $k_p,l_q$ are such that $\beta(k_p)\succ\beta(\l_q)$, then 
\begin{equation}\label{sw1}
 (k+1)_p\gg l_q\,\cdot
\end{equation}
Furthermore, if $k-l \geq \Delta(p,q)$, we then have either $\beta(k_{p})\succ \alpha(l_{q})$ or 
\begin{equation}\label{sw}
\alpha(l_{q})+1\gg \alpha((k-1)_{p})\succ \beta((k-1)_{p})\succ \beta(l_q)\,\cdot
\end{equation}
In the case of equality $k-l = \Delta(p,q)$, we necessarily have 
\begin{equation}\label{chaine}
 \beta(l_q)+1\succeq \beta(k_p)\,,
\end{equation}
and in the other case, we necessarily have that $\beta(k_{p})\succ \alpha(l_{q})$.
\end{lem}
\bi
\begin{lem}[\textbf{$1$-different-distance on $\Eee$}]\label{lem3}
Let us consider a partition $\nu =(\nu_1,\ldots,\nu_t)\in \Eee$. Then, for any $1\leq i<j\leq t$, we have
\begin{equation}\label{ee}
\nu_{i}\triangleright \nu_{j}+j-i-1\,\cdot
\end{equation}
\end{lem}
\bi
\section{Bressoud's algorithm}\label{sct3}
Here we adapt the algorithm given by Bressoud in his bijective proof of Schur's partition theorem \cite{BR80}. The mappings are easy to
describe and execute, but their justifications are more subtle and are
given in the next section.
\bi
\subsection{Machine $\Phi$: from $\Od$ to $\E$} Let us consider the following machine $\Phi$:
\begin{itemize}
\item[\Soo:]For a sequence $\la= \la_1,\ldots,\la_t$, take the smallest $i<t$ such that  $\la_i,\la_{i+1}\in \Pp$ and $\la_i\succ \la_{i+1}$ but $\la_i\not\gg \la_{i+1}$, if it exists, and replace
\begin{equation}
\begin{array}{l c l l}
\la_i &\leftarrowtail& \la_i+\la_{i+1}&\text{as a part in } \Sc\\
\la_j &\leftarrow& \la_{j+1}& \text{for all}\quad i<j<t\,\, 
\end{array}
\end{equation}
and move to \Stt. We call such a pair of parts a \textit{troublesome} pair. We observe that $\la$ loses two parts in $\Pp$ and gains one part in $\Sc$. 
The new sequence is $\la = \la_1,\ldots,\la_{t-1}$. Otherwise, exit from the machine.\\\\
\item[\Stt:]For $\la= \la_1,\ldots,\la_t$, take the smallest $i<t$ such that $(\la_i,\la_{i+1})\in \Pp\times\Sc$ and $\la_i\not\gg \la_{i+1}$ if it exists,  and replace 
\begin{equation}
(\la_i,\la_{i+1}) \looparrowright (\la_{i+1}+1,\la_i-1)\in \Sc\times\Pp
\end{equation}
and redo \Stt. We say that the parts $\la_i,\la_{i+1}$ are \textit{crossed}. Otherwise, move to \Soo.
\end{itemize} 
\bi
Let $\Phi(\la)$ be the resulting sequence after putting  any $\la=(\la_1,\ldots,\la_t)\in \Od$ in $\Phi$. 
This transformation preserves the size and the commutative product of primary colors of partitions.  
\begin{ex}
For $\C= \{a<b<c<d\}$, let us apply this machine on the partition $(5_b,3_d,2_a,1_d,1_c,1_b,1_a)$:
\begin{equation}\label{example1}
\begin{array}{ccccccccccc}
\begin{matrix}
5_b\\
3_d\\
\underline{2_a}\\
\underline{1_d}\\
1_c\\
1_b\\
1_a
\end{matrix} &\rightarrowtail&
\begin{matrix}
5_b\\
\mathbf{3_d}\\
\mathbf{3_{ad}}\\
1_c\\
1_b\\
1_a
\end{matrix}&
\looparrowright&
\begin{matrix}
5_b\\
4_{ad}\\
2_d\\
\underline{1_c}\\
\underline{1_b}\\
1_a
\end{matrix}&
\rightarrowtail&
\begin{matrix}
5_b\\
4_{ad}\\
\mathbf{2_d}\\
\mathbf{2_{bc}}\\
1_a
\end{matrix}&
\looparrowright&
\begin{matrix}
5_b\\
4_{ad}\\
3_{bc}\\
\underline{1_{d}}\\
\underline{1_a}
\end{matrix}&
\rightarrowtail&
\begin{matrix}
5_b\\
4_{ad}\\
3_{bc}\\
2_{ad}
\end{matrix}
\end{array}\,\cdot
\end{equation}
This example shows that $\Phi(\Od)\not\subseteq \Eee$.
\end{ex}
\bi
\subsection{Machine $\Psi$: on $\E$}
Let us consider the following machine $\Psi$:
\begin{itemize}
\item[\Soo:]For a sequence $\nu= \nu_1,\ldots,\nu_t$, take the greastest $i\leq t$ such that $\nu_i\in \Sc$ if it exists.  If $\nu_{i+1}\in\Pp$ and $\beta(\nu_i)\not\succ \nu_{i+1}$, then replace 
\begin{equation}
(\nu_i,\nu_{i+1}) \looparrowright (\nu_{i+1}+1,\nu_i-1)\in \Pp\times\Sc
\end{equation}
and redo \Soo. We say that the parts $\nu_i,\nu_{i+1}$ are \textit{crossed}. Otherwise, move to \Stt. If there are no more parts in $\Sc$, exit from the machine.\\\\
\item[\Stt:]For $\nu= \nu_1,\ldots,\nu_t$, take the the greatest $i\leq t$ such that $\nu_i\in \Sc$. By \Soo, it satisfies $\beta(\nu_i)\succ \nu_{i+1}$. Then replace
\begin{equation}
\begin{array}{l c l l}
\nu_{j+1} &\leftarrow& \nu_{j}& \text{for all}\quad t\geq j>i\,\, \\
(\nu_i) &\rightrightarrows& (\alpha(\nu_i),\beta(\nu_i))&\text{as a pair of parts in }\Pp \,,
\end{array}
\end{equation}
and move to \Soo.  We say that the part $\nu_i$ \textit{splits}. We observe that $\nu$ gains two parts in $\Pp$ and loses one part in $\Sc$. 
The new sequence  is $\nu = \nu_1,\ldots,\nu_{t+1}$.
\end{itemize}
\bi
Let $\Psi(\nu)$ be the resulting sequence after putting  any $\nu=(\nu_1,\ldots,\nu_t)\in \E$ in $\Psi$.
This transformation preserves the size and the product of primary colors of partitions.
\begin{exs}
For example, we choose $\C=\{a<b<c<d<e<f\}$ and we apply the machine $\Psi$ respectively on  $(4_{ae},3_{cd},3_{ab}), (4_a,3_{ae},2_{cd},1_b)$ and 
$(4_e,3_{ef},3_{cd},3_{ab},1_f)$, and we obtain 
\[ \begin{array}{ccccccccccc}
\begin{matrix}
4_{ae}\\
3_{cd}\\
2_a+1_b
\end{matrix} &
\rightrightarrows&
\begin{matrix}
4_{ae}\\
2_c+1_d\\
2_a\\
1_b
\end{matrix} &
\looparrowright&
\begin{matrix}
4_{ae}\\
3_a\\
1_d+1_c\\
1_b
\end{matrix}&
\rightrightarrows&
\begin{matrix}
2_e+2_a\\
3_a\\
1_d\\
1_c\\
1_b
\end{matrix}
&
\looparrowright&
\begin{matrix}
4_a\\
2_a+1_e\\
1_d\\
1_c\\
1_b
\end{matrix}
&
\rightrightarrows&
\begin{matrix}
4_a\\
2_a\\
1_e\\
1_d\\
1_c\\
1_b
\end{matrix}
\end{array}\,,
\]  
\[ \begin{array}{ccccc}
\begin{matrix}
4_a\\
3_{ae}\\
1_d+1_c\\
1_b
\end{matrix} &
\rightrightarrows&
\begin{matrix}
4_a\\
2_a+1_e\\
1_d\\
1_c\\
1_b
\end{matrix} &
\rightrightarrows&
\begin{matrix}
4_a\\
2_a\\
1_e\\
1_d\\
1_c\\
1_b
\end{matrix}
\end{array}\,,
\] 
\[
\begin{array}{ccccccccccccc}
\begin{matrix}
4_e\\
3_{ef}\\
3_{cd}\\
2_a+1_b\\
1_f
\end{matrix} &
\looparrowright&
\begin{matrix}
4_e\\
3_{ef}\\
3_{cd}\\
2_f\\
1_b+1_a
\end{matrix} &
\rightrightarrows&
\begin{matrix}
4_e\\
3_{ef}\\
2_c+1_d\\
2_f\\
1_b\\
1_a
\end{matrix}
\looparrowright&
\begin{matrix}
4_e\\
3_{ef}\\
3_f\\
1_d+1_c\\
1_b\\
1_a
\end{matrix}&
\rightrightarrows&
\begin{matrix}
4_e\\
2_e+1_f\\
3_f\\
1_d\\
1_c\\
1_b\\
1_a
\end{matrix}
\looparrowright&
\begin{matrix}
4_e\\
4_f\\
1_f+1_e\\
1_d\\
1_c\\
1_b\\
1_a
\end{matrix}&
\rightrightarrows&
\begin{matrix}
4_e\\
4_f\\
1_f\\
1_e\\
1_d\\
1_c\\
1_b\\
1_a
\end{matrix}
\end{array}\,\cdot
\]
With these examples, we can see that $\Psi$ is not injective on $\E$ and  $\Psi(\E)\not \subseteq \Od$.
\end{exs}
\section{Proof of \Thm{th3}}\label{sct4}
In this section, we prove \Thm{th3} by showing the following theorem.
\begin{theo}\label{thm6}
The transformation $\Phi$ describes an injection from $\Od$ into $\E$ such that $\Psi\circ\Phi_{|\Od} = Id_{|\Od}$.  
\end{theo}
\Thm{thm6} follows from the next three propositions whose proofs can be found in Section $9$.
\\In the following for any sequence $\mathcal{U} = u_1,\ldots,u_t$, we set $g(\mathcal{U})=u_1$ and $s(\mathcal{U})=u_t$ respectively the first and the last terms of $\mathcal{U}$.
\begin{prop}\label{pr1}
Let us consider any $\la=(\la_1,\ldots,\la_t)\in \Od$. Then, in the process $\Phi$ on $\la$, before the $u^{th}$ application of \Soo, 
there exists a triplet of partitions $(\delta^u,\gamma^u,\mu^u)\in \E\times (\E\cap\Od) \times \Od$ such that the sequence obtained is $\delta^u,\gamma^u,\mu^u$ and which satisfies the following conditions:
\begin{enumerate}
\item The $u^{th}$ application of \So occurs in the pairs $(s(\gamma^u),g(\mu^u))$, 
\item $s(\delta^{u})$ is the $(u-1)^{th}$ secondary part of $\delta^u$ and satisfies $s(\delta^u)\gg g(\gamma^u)$,
\item $\mu^{u+1}$ is the tail of the partition $\mu^{u}$ and has at least one fewer part than $\mu^{u}$.
\item $\delta^{u}$ is the head of $\delta^{u+1}$. 
\end{enumerate}     
\end{prop}
Note that the first triplet for $u=1$ has the form $(\emptyset,\gamma^1,\mu^1)$ with $(\gamma^1,\mu^1)\in(\E\cap\Od) \times \Od$ and $(s(\gamma^u),g(\mu^u))$ the first troublesome pair of $\la$.
The fact that $\Phi(\Od)\subset \E$ follows from \Prp{pr1} since $\mu^u$ strictly decreases in terms of number of parts and the process stops as soon as $\mu^{u}=\emptyset$. In fact, if $\mu^{u}\neq \emptyset$,
then $g(\mu^u)$ exists and we can still apply \So on the pair $(s(\gamma^u),g(\mu^u))$. The last triplet has then the form $(\delta^{S+1},\gamma^{S+1},\emptyset)$ with  $(\delta^{S+1},\gamma^{S+1})\in  \E\times (\E\cap\Od)$,
$s(\delta^{S+1})$ the $S^{th}$ and last secondary part of $\Phi(\la)$ and $s(\delta^{S+1})\gg g(\gamma^{S+1})$ if $\gamma^{S+1}\neq \emptyset$.
\begin{ex}
We again take the example $\la = (5_b,3_d,2_a,1_d,1_c,1_b,1_a)$ given in \eqref{example1}. We summarize the triplets of \Prp{pr1} in the following table:
\[
\begin{array}{|c|ccc|}
\hline
u&\delta^u&\gamma^u&\mu^u\\
\hline
1&\emptyset&5_b,3_d,2_a&1_d,1_c,1_b,1_a\\
\hline
2&5_b,4_{ad}&2_d,1_c&1_b,1_a\\
\hline
3&5_b,4_{ad},3_{bc}&1_d&1_a\\
\hline
4&5_b,4_{ad},3_{bc},2_{ad}&\emptyset&\emptyset\\
\hline
\end{array}\,\cdot
\]
\end{ex}
\begin{prop}\label{pr2}
Let us consider any $\nu=\nu_1,\ldots,\nu_t\in \E$. Then, in the process $\Psi$ on $\nu$, after the $(v-1)^{th}$ application of \Stt, there exists a triplet of partitions $(\delta^v,\gamma^v,\mu^v)$ with $\delta^v\in \E$ and
$\gamma^v,\mu^v$ some sequences of primary parts, such that the sequence obtained is $\delta^v,\gamma^v,\mu^v$ and which satisfies the following conditions:
\begin{enumerate}
\item $(s(\gamma^v),g(\mu^v))$ is the troublesome pair resulting from the $(v-1)^{th}$ splitting in \Stt,
\item $s(\delta^{v})\in \Sc$ so that the next iterations of \So after the $(v-1)^{th}$ \St occurs on this part,
\item $\mu^{v}$ is the tail of the sequence $\mu^{v+1}$ and has at least one fewer part than  $\mu^{v+1}$.
\item $\delta^{v+1}$ is the head of $\delta^{v}$.
\end{enumerate}    
\end{prop}
The process stops as soon as $\delta^v=\emptyset$, which means that we  have split every secondary part of $\nu$. If we set $S$ to be the number of secondary parts of $\nu$, the last triplet then
has the form $(\emptyset, \gamma^{S+1},\mu^{S+1})$ with $(s(\gamma^{S+1}),g(\mu^{S+1}))$ being a troublesome pair of primary parts. Also, we remark that the first triplet for $v=1$ is such that
$(\delta^1,\gamma^1,\emptyset)$ with $\delta^1$ equal to the head of $\nu$ up to the last secondary part, and with $\gamma^1$ equal to the tail of $\nu$ after this last part, so that $(\delta^1,\gamma^1)\in \E \times(\E\cap \Od)$ with $s(\delta^1)\gg g(\gamma^1)$ if $\gamma^1\neq \emptyset$.
\begin{ex}
We take the example $\nu = \Phi(\la)=5_b,4_{ad},3_{bc},2_{ad}$ in \eqref{example1}. We summarize the triplets of \Prp{pr2} in the following table:
\[
\begin{array}{|c|ccc|}
\hline
v&\delta^v&\gamma^v&\mu^v\\
\hline
1&5_b,4_{ad},3_{bc},2_{ad}&\emptyset&\emptyset\\
\hline
2&5_b,4_{ad},3_{bc}&1_d&1_a\\
\hline
3&5_b,4_{ad}&2_d,1_c&1_b,1_a\\
\hline
4&\emptyset&5_b,3_d,2_a&1_d,1_c,1_b,1_a\\
\hline
\end{array}\,\cdot
\]
\end{ex}
We now show that $\Psi\circ\Phi_{|\Od} = Id_{|\Od}$  using the following proposition.
\begin{prop}\label{pr10}
 For any $\la\in \Od$, if we set $\nu=\Phi(\la)$ and $S$ to be the number of secondary parts of $\nu$, then for any $v\in [1,S+1]$, the triplet of \Prp{pr2} is equal to the triplet of \Prp{pr1} for
 $u=S+2-v$.
\end{prop}
\section{Description of $\Ee = \Phi(\Od)$ and proof of \Thm{th4}}\label{sct5}
In this section, we set $\Ee=\Phi(\Od)$, and we give an explicit definition of the bridge for a partition $\nu\in \E$ in order to fit with the condition given in \Thm{th4}. 
Note that, by setting $\Ee=\Phi(\Od)$, the mapping $\Phi$ then describes a bijection between $\Od$ and $\Ee$, and $\Psi=\Phi^{-1}$, so that the identity \eqref{eq:final} holds and this implies \Thm{th4}. 
\[\]
\paragraph{\textbf{Enumeration of parts}} Let us consider a partition $\nu = (\nu'_1,\ldots,\nu'_{p+s})$ with $p$ primary parts and $s$ secondary parts. We can thus consider the $p+2s$ primary parts that occur in 
$\nu$ by counting both the upper and lower halves of the secondary parts. We then set
\begin{equation}
 \nu=(\nu_1,\ldots,\nu_{p+2s})
\end{equation}
with $J,I$ and $I+1$ defined to be respectively the sets of indices of the primary parts, the upper and lower halves of secondary parts. The secondary parts of $\nu$ are indeed  
the parts $\nu_i+\nu_{i+1}$ for $i\in I$. This method of enumeration according to the occurrences of the primary parts  was already used by the author in his proof of the generalization of Siladi\'c's theorem \cite{IK19}. We can then retrieve the corresponding indices for the parts $\nu'_k$ with
\begin{align*}
\nu_{j}&= \nu'_{j-|I\cap[1,j)|}\quad\text{for all }j\in J \,,\\
\nu_{i}+\nu_{i+1}&= \nu'_{i-|I\cap[1,i)|} \quad\text{for all }i\in I\,\cdot
\end{align*}
For ease of notation, we set $I=\{i_1<\cdots<i_s\}$ and $J=\{j_1<\cdots<j_p\}$.
We then consider the index set of the \textit{troublesome secondary parts} as defined in \eqref{eq:difdist}, 
\begin{equation}
 \ts(\nu) = \{i\in I: \nu^{-}(i)\triangleright\,\, \nu_i+\nu_{i+1}\nt\nu_{i+2}+\nu_{i+3} \}\,,
\end{equation}
where $\nu^{-}(i)=\nu'_{i-|I\cap[1,i]|}$ is the (primary or secondary) part to the left of $\nu_i+\nu_{i+1}$. We recall that, by \eqref{Ord} and \eqref{Ordd}, we do not have $\nu_i+\nu_{i+1}\triangleright\nu_{i+2}+\nu_{i+3}$ only if the pair of consecutive secondary parts has a pair of colors in $\Ccc$.
\begin{ex}\label{ex51}
We take $\nu=(14_{bd},11_a,10_{ad},9_{bc},8_{ac},3_c,2_{cd},2_{ab})\in \E$ with $ (p,s)=(2,6)$. Our enumeration gives
$$\nu=(\underbrace{7_d,7_d},11_a,\underbrace{5_d,5_a},\underbrace{5_b,4_c},\underbrace{4_c,4_a},3_c,\underbrace{1_d,1_c},\underbrace{1_b,1_a})\,$$ 
\begin{align*}
J=\{3,10\}, \quad I = \{1,4,6,8,11,13\},\quad I+1 = \{2,5,7,9,12,14\}\,,
\end{align*}
and $\ts(\nu) = \{4,11\}$.
\end{ex} 
\bi
We will then define, in the first part of this section, for any $i\in I$, the \textit{Bridge} $\Br(i)\geq i$ as an index in $I\cup J$, and the \textit{bridge} as the part $\nu_{\Br(i)}$ corresponding to this index. This definition will fit with the definition of $\Ee$ given in \Thm{th4}, that we can explicitly state in the following theorem.
\begin{theo}[\textbf{Explicit definition of $\Ee$}]\label{define}
The following are equivalent:
\begin{itemize}
 \item[$(1)$] $\nu \in \Ee = \Phi(\Od)$,
 \item[$(2)$] For any $i\in I$ such that $\Br(i)> i$, we have 
\[\nu^{-}(i)\gg\,\,\nu_{\Br(i)}+\frac{\Br(i)-i}{2}\not\succ \,\,\nu_i+\nu_{i+1}\,,\]
 \item[$(3)$]\textbf{(Necessary and sufficient checks)} For all $i\in \ts(\nu)$ such that $\Br(i)>i$, we have 
 \begin{equation}\label{final}
  \nu_i+\nu_{i+1}\succ \nu_{\Br(i)}+\frac{\Br(i)-i}{2}\,\cdot
 \end{equation}
\end{itemize}
\end{theo}
Recall that if $\nu\in \Eee$, then $\ts(\nu)=\emptyset$ so that $(3)$ is true. We thus retrieve the fact that $\Eee\subset \Ee$.
\\\\In the remainder of this section, we will first give an explicit definition of the bridge, describe its properties and show how to easily compute it. Then, in the next part, we will prove that $(1)$ implies $(2)$. After that, we show that $(2)$ implies $(1)$. 
Finally, we give a proof of the equivalence between $(2)$ and $(3)$.
\subsection{Definition and properties of the Bridge}
 For any $i\in I$, let us consider $j = \min(i,p+2s]\cap J$, if it exists,
which is the index of the greatest primary part to the right of the secondary part $\nu_i+\nu_{i+1}$. Otherwise, there is no primary part to its right, and we set $j=p+2s+1$.
Note that $j-i$ is twice the number of secondary parts ($\nu_i+\nu_{i+1}$ included) between $\nu_i+\nu_{i+1}$ and $\nu_j$, even if we set $\nu_{p+2s+1}=0_{a_n}$.
In any case, we can set $j = \min(i,p+2s+1]\cap( J\cup\{p+2s+1\})$.
\begin{deff} We define the \textit{Bridge} $\Br(i)$ to be as follows :
\begin{itemize}
 \item If  $j$ satisfies 
 \begin{equation}\label{br1}
  \nu_{i'+1}\not \succ \nu_j+\frac{j-i'}{2}-1
 \end{equation}
 for all $i'\in [i,j)\cap I$, we set $\Br(i)=j$. Note that for $j=p+2s+1$, the relation \eqref{br1} is never satisfied for the last secondary part, since its upper and lower halves have size greater than $0$. 
\item Otherwise, we define 
\begin{equation}\label{br2}
 \Sc_i=\{u\in (i,j)\cap I: \nu_{i'+1}\not \succ \nu_u +\frac{u-i'}{2}-1\quad \forall  i'\in [i,u)\cap I\}.
\end{equation}
If $\Sc_i\neq \emptyset$, we then set 
\begin{equation}\label{br3}
 \Br(i)= \max{\Sc_i}\,\cdot 
\end{equation}
Otherwise, we set $\Br(i)=i$.
\end{itemize}
\end{deff}
Here, we observe that $\Br(i)\geq i$, and for $\Br(i)> i$, we have the relation
\begin{equation}\label{cross}
 \nu_{i'+1}\not \succ \nu_{\Br(i)} +\frac{\Br(i)-i'}{2}-1
\end{equation}
for all $i'\in [i,\Br(i))\cap I$. Also note that the function $\Br$ is \textit{local}, as it only depends on the maximal sequence of secondary parts and not on the entire partition $\nu$. 
\begin{rem}
The value $\frac{\Br(i)-i'}{2}$ indeed corresponds to the
difference between the index of the secondary part $\nu'_{i'-|I\cap [1;i')|}$ and the index of the primary or secondary part $\nu'_{\Br(i)-|I\cap [1;\Br(i))|}$, so that the relation \eqref{cross} can be formulated as follows:
the lower half $\nu_{i'+1}$is not $1$-distant-different from $\nu_{\Br(i)}-1$.
\end{rem}
\bi
The definition of brigde as stated above has the sole purpose to  make our results simpler to prove.  
It may seem difficult to compute, but the calculation of the bridge is indeed quite simple as it can be done recursively. In fact, the first hint for the computational method is given by the following lemma.  
\begin{lem}\label{lem4}
The function $\Br$ is non-decreasing on $I$, and for any $i$ such that $\Br(i)\in I$, we have
$\Br(\Br(i))=\Br(i)$.
\end{lem}
\Lem{lem4} allows us to state that for any $i\in I$, $\Br(i)$ is either the index of the greastest primary part to the right of $\nu_i+\nu_{i+1}$, or the smallest fixed point (by $\Br$) to its right.
This fact leads to the following proposition, which gives us the second and final hint for the computation of $\Br$.
\begin{prop}[\textbf{Crossing rules for $\Psi$}]\label{pr5}
By applying $\Psi$ on $\nu = (\nu_1,\ldots, \nu_{p+2s})$, we have that the secondary part $\nu_i+\nu_{i+1}$:
 \begin{itemize}
  \item does not cross any primary part if and only if $\Br(i)=i$,
  \item otherwise, for $i_u=i<\Br(i)$, it first crosses the primary part that comes from $\nu_{\Br(i)}$: 
  \begin{equation}
  g(\gamma^{s+1-u})=\nu_{\Br(i_u)} +\frac{\Br(i_u)-i_u}{2}-1 \,\cdot
  \end{equation}
 \end{itemize}
\end{prop}
The relevance of this proposition consists in saying that, during $\Psi$, the fixed points are the indices of the secondary parts which split directly with no application of \Soo, and if a fixed point $i=\Br(i)$ is found, then the next fixed point to its left is the index of the smallest secondary part which is not crossed by the upper half $\nu_i$ during iterations of \Soo. 
\\Note that, by definition, the \textit{bridges} are exactly the parts $\nu_i$ for the fixed points $i$, along with the primary parts $\nu_j$ after the tail of a sequence of secondary parts.  The key idea to compute the bridge is then to retrieve the fixed points by performing iterations of \So with the bridges $\nu_j$ and $\nu_i$. 
\bi
\paragraph{\textbf{Method to compute $\Br$}} The function $\Br$ being local, we then consider a maximal sequence of secondary parts, with the ending primary part to its right. The reasoning will be the same when we do not have a primary part at the tail of the sequence. Without loss of generality, we can restrict the partition $\nu$ to such sequence: $\nu  = (\nu_1,\ldots,\nu_{2s+1})$ with  
\[\nu_1+\nu_2\,\,\,\gg\nu_3+\nu_4\,\gg\,\cdots\,\,\gg\,\, \nu_{2s-1}+\nu_{2s}\gg\,\, \nu_{2s+1}\,\cdot\]
For simplicity, we show the computation on the following example. We take the set of primary colors $\C=\{a<b<c<d<e<f\}$ and the partition
\[\nu = (20_{ef},20_{ad},19_{bc},16_{de},14_{af},11_{ad},6_c)\,,\]
or rewritten with our enumeration 
\[\nu = (\underbrace{10_f,10_e}_{i=1}\,,\,\underbrace{10_d,10_a}_{i=3}\,,\,\underbrace{10_b,9_c}_{i=5}\,,\,\underbrace{8_e,8_d}_{i=7}\,,\,\underbrace{7_f,7_a}_{i=9}\,,\, \underbrace{6_a,5_d}_{i=11}\,,\,\underbrace{6_c}_{j=13})\,\cdot\]
Recall that to perform \So of $\Psi$, we always compare a primary part to the lower half of a secondary part. We then proceed as follows:
\begin{enumerate}
\item We start with the sequence 
$$(\beta_1,\beta_2,\cdots,\beta_{s},\alpha_{s+1})=(\nu_2,\nu_4,\cdots,\nu_{2s},\nu_{2s+1})$$ consisting of  the lower halves and the primary part. Our example gives the sequence
\[(\underbrace{10_e,10_a,9_c,8_d,7_a,5_d}_{\beta_u , u=1,\ldots,6},\underbrace{6_c}_{\alpha_7})\,\cdot\]
The first fixed point (starting from the right) corresponds to the first $\beta_{u}$ which is $1$-different-distant from $\alpha_{s+1}-1$ in the order $\succ$. We then have  $i_1=2u_1-1$ if such $u_1$ exists. If there is no such $u_1$, it means that $j$ is the Bridge of all $i\in 2\sss-1$.
With our example, we just have to compare the two sequences
\begin{align*}
(10_e,10_a,9_c,\underline{8_d},7_a,5_d)\\
(11_c,10_c,9_c,8_c,7_c,6_c)
\end{align*}
starting from the right, and we identify the first fixed point, $i_1=2u_1-1=7$, corresponding to the underlined lower half.
\item We redo the same process for the sequence 
$$(\beta_1,\beta_2,\cdots,\beta_{u_1-1},\alpha_{u_1})=(\nu_2,\nu_4,\cdots,\nu_{i_1-1},\nu_{i_1})\,,$$
where $\beta_u$ are the lower halves of the $(u_1-1)$ first secondary parts, and $\alpha_{u_1}$ is the upper half the $u_1^{th}$ secondary part, which corresponds to the first Bridge. Our example gives the sequence 
$(\underbrace{10_e,10_a,9_c}_{\beta_{1,2,3}},8_e)$ and 
the sequence comparison 
\begin{align*}
(10_e,10_a,\underline{9_c})\\
(10_e,9_e,8_e)
\end{align*}
and the second fixed point is $i_2=2u_2-1=5$.
\item Following the same process, we apply the comparisons for the sequence 
$$(\beta_1,\beta_2,\cdots,\beta_{u_k-1},\alpha_{u_k})=(\nu_2,\nu_4,\cdots,\nu_{i_k-1},\nu_{i_k})\,,$$
in order to retrieve the $(k+1)^{th}$ fixed point. Here again, we have $i_k=2u_k-1$. If there is no $\beta_{u}$ which is $1$-different-distant from $\alpha_{u_k}-1$ in the order $\succ$, we  stop the process, as $i_k$ is the last fixed point and becomes the Bridge of the remaining $i<i_k$. In our example the last fixed point is indeed $i_2$, since 
we have the sequence  $(\underbrace{10_e,10_a}_{\beta_{1,2}},10_b)$ and the sequence comparison
\begin{align*}
(10_e,10_a)\\
(11_b,10_b)\cdot
\end{align*}
\end{enumerate}
Note that applying this computation requires in fact $s$ comparisons, starting from the right to the left, to retrieve all of the fixed points, but computing the precise bridge for an $i$ will require as many comparisons as the number of secondary parts to its right.
For our example, we sum up the computation of the Bridge with the following table.
\begin{equation}
\begin{array}{|c|c|c|c|c|c|c|}
\hline
 i&1&3&5&7&9&11\\
 \hline
 \Br(i)&5&5&5&7&13&13\\
 \hline
\end{array} \,\cdot            
\end{equation} 
\bi By condition $(3)$ of \Thm{define}, to see if $\nu \in \Ee$, we only need to check the secondary part $20_{ef}$, whose bridge corresponds to $10_b$, and we have $20_{ef}\succ 10_b+2$. We then have $\nu\in \Ee$. One can check that 
\[\Psi(\nu) = (12_b,11_a,9_f,9_e,9_d,9_c,8_e,8_d,8_c,7_a,6_f,5_d,5_a)\,,\]
and that $\Phi(\Psi(\nu))=\nu$.
\bi  For the case where the sequence $\nu=(\nu_1,\ldots,\nu_{2s})$ does not end by a primary part, the first splitting occurs at the right most secondary part, and we set the first fixed point $i_1=2u_1-1=2s-1$. We then start the process at step $(2)$ and the remainder of the computation of the bridges is the same.
\subsection{Proof that $(1)$ implies $(2)$}
We suppose that $i=i_{s+1-v}$ for some $v\in [1,s]$. Then by the \Prp{pr5} and \Prp{pr2}, $\nu_i+\nu_{i+1} = s(\delta^{v})$ and $g(\gamma^v)=\nu_{\Br(i)} +\frac{\Br(i)-i}{2}-1$.
After crossing, the primary part becomes $\nu_{\Br(i)} +\frac{\Br(i)-i}{2}$ and the secondary part becomes $\nu_i+\nu_{i+1}-1$. 
But, by \Prp{pr10}, the crossing is the reverse crossing of \St in process $\Phi$, so that we have 
\[\nu_{\Br(i)} +\frac{\Br(i)-i}{2} \not \gg \nu_i+\nu_{i+1}-1 \Longleftrightarrow \nu_{\Br(i)} +\frac{\Br(i)-i}{2}\not \succ \nu_i+\nu_{i+1}\,\cdot\]
Also, note that the sequence 
\[\delta^v \setminus \{\nu_i+\nu_{i+1}\}\quad ,\quad \nu_{\Br(i)} +\frac{\Br(i)-i}{2} \]
is indeed the head of the sequence $\delta^{v+1},\gamma^{v+1}$, which is a partition in $\E$ by \Prp{pr10}. In fact, this pair of sequences corresponds to the same pair in \Prp{pr1} for $u=s-v$, and is a pair in $\E\times (\E\cap \Od)$
satisfying $s(\delta^u)\gg g(\gamma^u)$. We then deduce that the part $\nu^-(i)$ to the left $\nu_i+\nu_{i+1}$ is well-ordered with $\nu_{\Br(i)} +\frac{\Br(i)-i}{2}$ in terms of $\gg$, so that 
\[\nu^-(i)\gg \nu_{\Br(i)} +\frac{\Br(i)-i}{2}\,\cdot\]
With this, we have proved that $(1)$ implies $(2)$ in \Thm{define}.
\subsection{Proof that $(2)$ implies $(1)$} 
We prove that $(2)$ implies $(1)$ with the following proposition.
\begin{prop}\label{pr6}
If $\nu$ satisfies condition $(2)$ in \Thm{define}, then in \Prp{pr2}, the triplet $(\delta^v,\gamma^v,\mu^v)$ satisfies the following properties:
\begin{enumerate}
\item $(\gamma^v,\mu^v)\in (\E\cap\Od)\times \Od$,
\item $s(\delta^v)\gg g(\gamma^v)$.
\item If we apply \So once and  some iterations of \St of the process $\Phi$   on the sequence $\delta^{v+1},\gamma^{v+1},\mu^{v+1}$, we obtain the sequence $\delta^v,\gamma^v,\mu^v$
\end{enumerate} 
\end{prop}
\Prp{pr6} says that, for any $\nu\in \E$ that satisfies $(2)$ of \Thm{define}, we have that $\Psi(\nu)\in \Od$, since the last sequence $\delta^{S+1},\gamma^{S+1},\mu^{S+1}$ is such that $\delta^{S+1} =\emptyset$
and $(s(\gamma^v),g(\mu^v))$ is a troublesome pair so that $s(\gamma^v)\succ g(\mu^v)$.
The fact that all the crossings and the splitting of $\Psi$ are reversible by $\Phi$ means that the process $\Psi$ on $\nu$ is reversible by $\Phi$,
and we then have $\Ee\ni\Phi(\Psi(\nu))=\nu$.
\subsection{Proof of the equivalence between $(2)$ and $(3)$}
In this part, we will show that it sufficient to satisfy the condition $(2)$ only on $\ts(\nu)$. In fact, condition $(2)$ of \Thm{define} implies that \eqref{final} is true on $\ts(\nu)$, so that 
$(2)$ implies $(3)$.
To prove that $(3)$ implies $(2)$, we will use the following lemmas.
\begin{lem}\label{lem5}
 Let us consider some consecutive secondary parts $\nu_i+\nu_{i+1}\gg \cdots\gg\nu_{i'}+\nu_{i'+1}$ such that 
\[
  \nu_i+\nu_{i+1}\nt\cdots\nt\nu_{i'}+\nu_{i'+1}\,\cdot
\]
 We then have that 
 \begin{equation}\label{X}
  \nu_{i'}+\nu_{i'+1}+ \frac{i'-i}{2}\succ \nu_i+\nu_{i+1}\,\cdot
 \end{equation}
\end{lem}
\begin{lem}\label{lem6}
 Let us consider some consecutive secondary parts $\nu_i+\nu_{i+1}\gg \cdots\gg\nu_{i'}+\nu_{i'+1}$ such that 
the size differences between consecutive parts are minimal. If $\Br(i')>i'$, then $\Br(i)=\Br(i')$.
\end{lem}
\begin{proof}[Proof that $(3)$ implies $(2)$]
 Let us consider a maximal sequence of consecutive secondary parts $\nu_i+\nu_{i+1}\gg \cdots\gg\nu_{i'}+\nu_{i'+1}$ with
 \[
  \nu_i+\nu_{i+1}\nt\cdots\nt\nu_{i'}+\nu_{i'+1}\,\cdot
 \] We then have that the extremal parts are well-ordered in terms of $\triangleright$ with the parts to the left and to the right of the sequence, and  
 we have the inequality
 \begin{equation}\label{seq}
 \cdots \triangleright\nu_i+\nu_{i+1}\nt\cdots\nt\nu_{i'}+\nu_{i'+1}\triangleright\cdots
 \end{equation}
In particular, $i\in \ts(\nu)$. Now, let us consider the set 
\[
 \{u\in [i,i']\cap I: \Br(u)>u\}\,\cdot
\]
If it is empty, then any $u\in [i,i']\cap I$  is a fixed-point of $\Br$. Otherwise, by \Lem{lem6}, it has the form $[i,u]\cap I$ and $\Br$ is the identity on $(u,i']\cap I$.
Furthermore, $\Br(i)=\Br(u')>u'$ for all $u'\in [i,u]\cap I$.
\\\\If we assume that  
\[
 \nu_i+\nu_{i+1} \succ \nu_{\Br(i)}+\frac{\Br(i)-i}{2}\,,
\]
 by \eqref{X}, we then have for all $u'\in [i,u]\cap I $
\[
 \nu_{u'}+\nu_{u'+1} \succ \nu_{\Br(u')}+\frac{\Br(u')-u'}{2} \Longleftrightarrow  \nu_{\Br(u')}+\frac{\Br(u')-u'}{2}\not\succ \,\,\nu_{u'}+\nu_{u'+1} \,\cdot
\]
In addition, by \eqref{Ordd}, we obtain,  for all $u'\in (i,u]\cap I$, that we also have $u'-2\in [i,u)\cap I$ and have $\Br(u'-2)=\Br(u')$, and the following
\[\nu_{u'-2}+\nu_{u'-1}\succ \nu_{\Br(u'-2)}+\frac{\Br(u'-2)-u'+2}{2}\quad
\Longleftrightarrow\quad  \nu_{u'-2}+\nu_{u'-1}\gg \nu_{\Br(u')}+\frac{\Br(u')-u'}{2}\,,\]
so that the condition $(2)$ is also satisfied. Note that condition $(2)$ is also satisfied in $i$, since we have by definition \eqref{Ord}
\begin{align*}
\nu^{-}(i) \triangleright \nu_i+\nu_{i+1} \succ \nu_{\Br(i)}+\frac{\Br(i)-i}{2} &\Longrightarrow \nu^{-}(i)\triangleright \nu_{\Br(i)}+\frac{\Br(i)-i}{2}\not \succ \nu_i+\nu_{i+1}\\
&\Longrightarrow \nu^{-}(i)\gg \nu_{\Br(i)}+\frac{\Br(i)-i}{2}\not \succ \nu_i+\nu_{i+1}\,\cdot
\end{align*}
We then have proved that the condition $(2)$ is satisfied for any element of $I$ in a sequence of the form \eqref{seq}. 
\bi
Now let us take $i\in I$ such that $i$ is not in a sequence of the form \eqref{seq}. This is equivalent to saying that $\nu_i+\nu_{i+1} $ 
is well-ordered to its left and to its right in terms of $\triangleright$, so that
\[
 \cdots \triangleright \nu_i+\nu_{i+1} \triangleright \cdots \cdot
\]
We can then see by \eqref{Ord} that, for $\Br(i)>i$,  
\begin{align*}
\nu^{-}(i) \triangleright \nu_i+\nu_{i+1} \succ \nu_{\Br(i)}+\frac{\Br(i)-i}{2} &\Longrightarrow \nu^{-}(i)\triangleright \nu_{\Br(i)}+\frac{\Br(i)-i}{2}\not \succ \nu_i+\nu_{i+1}\\
&\Longrightarrow \nu^{-}(i)\gg \nu_{\Br(i)}+\frac{\Br(i)-i}{2}\not \succ \nu_i+\nu_{i+1}\,\cdot
\end{align*}
This means that we only need to prove that $\nu_i+\nu_{i+1} \succ \nu_{\Br(i)}+\frac{\Br(i)-i}{2}$ in order to satisfy the condition $(2)$.
\begin{itemize}
 \item Suppose first that there exists $i'\in \ts(\nu)$ such that $i'\in (i,\Br(i))$. We then have by \Lem{lem4} that $\Br(i')=\Br(i)$. By taking $i'$ the minimum of all such elements, we than have the sequence
 \[\nu_i+\nu_{i+1}\triangleright \cdots \triangleright \nu_{i'}+\nu_{i'+1}\]
 so that, by \eqref{Ord} and the fact that the parts between these two are in $\Sc$, we obtain
 \[\nu_i+\nu_{i+1}\succ \nu_{i'}+\nu_{i'+1} + \frac{i'-i}{2}\,\cdot\]
 Since $i'$ satisfies condition $(3)$, we then have 
 \[\nu_{i'}+\nu_{i'+1} \succ \nu_{\Br(i')}+\frac{\Br(i')-i'}{2}\,,\]
 and thus, 
 \[
  \nu_{i}+\nu_{i+1} \succ \nu_{\Br(i)}+\frac{\Br(i)-i}{2}\,\cdot
 \]
\item If $(i,\Br(i))\cap\ts(\nu) =\emptyset$, we then have  the sequence
\[\nu_i+\nu_{i+1}\triangleright \cdots \triangleright \nu_{\Br(i)-2}+\nu_{\Br(i)-1}\triangleright \nu_{\Br(i)} \]
if $\Br(i)\in J$, and otherwise,  
\[\nu_i+\nu_{i+1}\triangleright \cdots \triangleright \nu_{\Br(i)-2}+\nu_{\Br(i)-1}\triangleright \nu_{\Br(i)}+\nu_{\Br(i)+1}\,\cdot\]
By \eqref{Ord}, in the first case, we directly have 
\[
 \nu_{i}+\nu_{i+1} \succ \nu_{\Br(i)}+\frac{\Br(i)-i}{2}\,,
\]
while in the second case, we obtain
\[
  \nu_{i}+\nu_{i+1} \succ \nu_{\Br(i)}+\nu_{\Br(i)+1}+\frac{\Br(i)-i}{2}\,\cdot
\]
But, in terms of part sizes for the second case, we have by definition \eqref{cons1} that 
\[\nu_{i}+\nu_{i+1} -\left(\nu_{\Br(i)}+\frac{\Br(i)-i}{2}\right)\geq \nu_{\Br(i)+1}\geq 1\,,\]
so that, again by \eqref{cons1},
\[
 \nu_{i}+\nu_{i+1} \succ \nu_{\Br(i)}+\frac{\Br(i)-i}{2}\,\cdot
\]
\end{itemize}
\end{proof}
\bi
\section{Forbidden patterns of $\Ee$}\label{sct6}
In this section, we study the forbidden patterns that a partition in $\E$ has to avoid to be in $\Ee$.
\bi 
By the definition of the \textit{bridge} and \Thm{define}, we can see that the reversibility of $\Psi$ by $\Phi$ is a local problem. 
In fact, for any secondary part in a partition $\nu\in \E$, the reversibility only depends on the sequence starting from this part up to 
either the greatest primary part to its right if it exists, or the last part of $\nu$ if there is no primary part to its right.  
Furthermore, by condition $(3)$ of \Thm{define}, we only have to consider the sequences whose head is a sequence which is not well-ordered by $\triangleright$. 
Then, it suffices to restrict the fordidden patterns to those such that the first part does not satisfy \eqref{final}:
\begin{equation}
 \nu = \nu_1+\nu_2\,\,\, \nt\nu_3+\nu_4\,\gg\,\cdots\,\, \gg\,\, \nu_{2s+1} \quad \text{or}\quad \nu_{2s+1}+\nu_{2s+2}\,,
\end{equation}
such that $\Br(1)=2s+1$ and $\nu_{2s+1} +s \succ \nu_1+\nu_2$.
\begin{rem}\label{rem1}
It is sufficient to consider the last part to be a primary part. In fact, a sequence that ends by a secondary part can be viewed as the same sequence with this last part replaced by its upper half, as by \eqref{lex} and \eqref{Ordd},  
\begin{align*}
\nu_{2s-1}+\nu_{2s}\gg\nu_{2s+1}+\nu_{2s+2}&\Longrightarrow \nu_{2s-1}+\nu_{2s}\succ\nu_{2s+1}+\nu_{2s+2}\\
&\Longrightarrow \nu_{2s-1}+\nu_{2s}\succ\nu_{2s+1}+1\\
&\Longrightarrow \nu_{2s-1}+\nu_{2s}\gg\nu_{2s+1}\,\cdot
\end{align*}
\end{rem}
Note that, if a pattern $\nu$ is forddiden, then any pattern $\eta$ whose head or tail is $\nu$ is also forddiden. This is obvious when the tail of $\eta$ is $\nu$ since the troublesome crossing will not change.
When $\nu$ is the head of $\eta$, we have that $\mathbf{Br}_{\eta}(1) = \mathbf{Br}_{\eta}(\Br(1))$ 
and we use the same reasoning as in the proof of 
\Lem{lem3} given in Section 8 to show that 
\[\nu_{\Br(1)}+\frac{\Br(1)-1}{2}\succ \nu_1+\nu_2 \Longrightarrow \eta_{\mathbf{Br}_{\eta}(1)}+ \frac{\mathbf{Br}_{\eta}(1)-1}{2}\succ \eta_1+\eta_2\,\cdot\]
Therefore, the \textit{optimal} fordidden patterns are the ones that are allowed after removing either the first part or the last part. Furthermore, these forddiden patterns satisfy the fact that 
the Bridge of the first part is the position of the last part, so that during the process of $\Psi$, every secondary part is crossed by the last part if it is a primary part, or by its upper half. 
The optimization also implies that all these crossings are reversible by $\Phi$, except the last one which occurs with the first part of the pattern. 
\bi
In the next subsections, we first give some particular properties of the optimal forbidden patterns, and after that, we aim at retrieving the optimal forbidden patterns for four primary colors. Finally, we 
enumerate the optimal forbidden patterns, with some restrictions, for five primary colors, showing that there is an infinitude of optimal forbidden patterns for more than four primary colors.
\bi
\subsection{Properties of optimal forbidden patterns}
We first define a tool that will help to have a better understanding of the optimal forbidden patterns.
\begin{deff}
We say that two secondary colors $p$ and $q$ are \textit{primary equivalent} if and only if their orders according to the primary colors are the same, which means that
$p=a_ia_u$ and $q=a_ia_v$ for some $u,v\in (i,n]$. We then use the notation $k_p\equiv k_q$ and the equivalence class $\overline{k_p}$. This matters in the sense that
for any primary color $c$, we have the equivalence between $k_p\equiv k_q$ and 
\begin{equation}\label{compare}
 k_p\succ l_c \Longleftrightarrow k_q \succ l_c\,\cdot
\end{equation}
We can then write $\overline{k_p}\succ l_c$. For two secondary colors $p$ and $q$, we say that $\overline{k_p}\succ\overline{h_q}$ if and only if we can find a primary part $l_c$ such that
$\overline{k_p}\succ l_c \succ \overline{h_q}$.  This is equivalent to saying that $k>h$ or $k=h$ and $(p,q)=(a_ia_u,a_ja_v)$ with $i>j$.
\end{deff}
\bi
Let us now consider an optimal forbidden pattern
\begin{equation}
 \nu= \nu_1+\nu_2\,\,\, \nt\nu_3+\nu_4\,\gg\,\cdots\,\, \gg\,\, \nu_{2s+1}
\end{equation}
where the secondary parts are $\nu_{2i-1}+\nu_{2i}$ and the last part $\nu_{2s+1}$ is a primary part. In the remainder of the section, we consider the different-distance with respect to the order $\succ$.
We thus have the following properties:
\begin{enumerate}
 \item For all $i\in[1,s]$, we have $\Br(2i-1)= 2s+1$.
 \item The part $\nu_{2s+1}$ is $1$-different-distant from $\overline{\nu_{1}+\nu_{2}}$:
 \begin{equation}\label{eq:crossed}
  \nu_{2s+1}+s\succ \overline{\nu_{1}+\nu_{2}}\,,
 \end{equation}
 \item The fact that the pattern $\nu_3+\nu_4\gg\cdots\gg\nu_{2s-1}+\nu_{2s}\gg\nu_{2s+1}$ is  allowed implies by \Thm{define}, for all $i\in[2,s]$, that $\overline{\nu_{2i-1}+\nu_{2i}}$ is $1$-different-distant from $\nu_{2s+1}$,
 \begin{equation}\label{bet}
\overline{\nu_{2i-1}+\nu_{2i}}\succ \nu_{2s+1}+s+1-i\,,
 \end{equation}
 and by transitivity, this implies that $\overline{\nu_{2i-1}+\nu_{2i}}$ is $1$-different-distant from $\overline{\nu_1+\nu_2-i+1}$,
 \begin{equation}\label{loop}
  \overline{\nu_{2i-1}+\nu_{2i}}\succ \overline{\nu_1+\nu_2-i+1}\,\cdot
 \end{equation}
\item We obtain the following inequality
\begin{equation}\label{between}
\overline{\nu_3+\nu_4+1} \succ \nu_{2s+1}+s \succ \overline{\nu_1+\nu_2}\,\cdot
\end{equation}
\item If we replace the primary part $\nu_{2s+1}$ by another $\nu'_{2s+1}$ satisfying $\nu_1+\nu_2\succ\nu'_{2s+1}+s$, we then obtain the following allowed pattern 
\[
 \nu'= \nu_1+\nu_2\nt \nu_3+\nu_4\gg\cdots\gg\nu_{2s-1}+\nu_{2s}\gg \nu'_{2s+1}\,\cdot
\]
\end{enumerate}
\begin{rem}
By \eqref{rem1}, a pattern $\nu_1+\nu_2\gg\cdots\gg\nu_{2s-1}+\nu_{2s}\gg \nu_{2s+1}+\nu_{2s+2}$ only consisting of secondary parts is optimal and forbidden if and only if 
$\nu_1+\nu_2\gg\cdots\gg\nu_{2s-1}+\nu_{2s}\gg \nu_{2s+1}$ is an optimal forbidden pattern. Note that in this case, \eqref{loop} is also satisfied for $i=s+1$. 
\end{rem}
\bi
We now define a special kind of pattern, that we call a \textit{shortcut}.
\begin{deff}
 A pattern $\nu_1+\nu_2\gg\cdots\gg\nu_{2s+1}+\nu_{2s+2}$ is said to be a shortcut if 
\begin{equation}\label{kill}
\overline{\nu_{2s+1}+\nu_{2s+2}}\succ \overline{\nu_1+\nu_2-s+1}\cdot         
\end{equation}
\end{deff}
One can check that a shortcut has at least three secondary parts, and that the relation \eqref{kill} is stronger than \eqref{loop}. The following property makes the enumeration of optimal forbidden patterns 
which contain shorcuts quite difficult. 
\begin{prop}\label{pr7}
We can always build a forbidden pattern starting from any allowed pattern and using iteration of a shortcut.
\end{prop}
\bi
By considering the optimal forbidden pattern $\nu= \nu_1+\nu_2\,\,\, \nt\nu_3+\nu_4\,\gg\,\cdots\,\, \gg\,\, \nu_{2s+1}$ which does not contain any shortcut, we then have by \eqref{eq:crossed},\eqref{bet} and \eqref{kill} the following relation for all $i\in \{1,\ldots,s-1\}$:
\begin{equation}\label{arc1}
\overline{\nu_1+\nu_2-i+1}\succeq \overline{\nu_{2i+1}+\nu_{2i+2}}\succ \nu_{2s+1}+s-i\succ\overline{\nu_1+\nu_2-i}\,\cdot
\end{equation}
The latter implies the following properties:
\begin{enumerate}
\item By definition of the head and \eqref{Ordd},  $\nu_1+\nu_2$ and $\nu_3+\nu_4$ are consecutive for $\succ$.\\
\item For all $i\in \{2,\ldots,s-1\}$, two consecutive parts $\nu_{2i-1}+\nu_{2i}$ and $\nu_{2i+1}+\nu_{2i+2}$ are either consecutive in terms of $\succ$ (or equivalently not well-ordered by $\triangleright$), or consecutive in terms of $\triangleright$. In fact, by \eqref{arc1}, we necessarily have 
$$\overline{\nu_{2i+1}+\nu_{2i+2}+2}\succ \overline{\nu_{2i-1}+\nu_{2i}}\,\,\Longrightarrow \,\,\nu_{2i-1}+\nu_{2i}\not\succeq \nu_{2i+1}+\nu_{2i+2}+2 \,\cdot$$
\item By \eqref{arc1}, we have 
$$\nu_{2s+1}+2\succ \overline{\nu_1+\nu_2-s+2}\succeq \overline{\nu_{2s-1}+\nu_{2s}}\succ \nu_{2s+1}+1\,,$$
so that, by \eqref{Ord}, $\nu_{2s-1}+\nu_{2s}$ and $\nu_{2s+1}$ are consecutive for $\triangleright$.
\end{enumerate} 
\bi
We see that the optimal forbidden patterns with no shortcut have their parts either consecutive in the order $\succ$ or in the order $\triangleright$.
Let us then consider the following moves: 
\begin{itemize}
 \item The arrow $p \textcolor{blue}{\rightarrow} q$ means that $(p,q)$ is a special pair and 
it represents a pattern of the form \[(k+\chi(p\leq q))_p,k_q\,\cdot\]
\item The two-headed arrow $p \textcolor{red}{\twoheadrightarrow} q$ represents a move from a part with color $p$ to the greatest secondary part with color $q$
 smaller than the first part in terms of $\triangleright$. In fact, it indeed represents the pattern \[k+1+\chi(p\leq q))_p, k_q\,\cdot\]
\end{itemize}
Therefore, the optimal forbidden patterns with no shortcut have the form
\begin{equation}
 c_1 \circ\cdots \circ c_m \quad , \quad k
\end{equation}
where $c_1,\ldots,c_m$ are some colors, $\circ$ is either $\rightarrow$ or $\twoheadrightarrow$, and $k$ is the size of the smallest part, so that the last part is $k_{c_m}$.
\begin{ex}
For $\C=\{a<b<c<d\}$ , the pattern
\[ad\rightarrow bc \twoheadrightarrow cd \twoheadrightarrow b\quad , \quad 5\]
will represent the pattern $9_{ad},8_{bc},6_{cd},5_b$.
\end{ex}
Since an optimal forbidden pattern is allowed after removing the last part, we will consider the following form 
\begin{equation}\label{opt}
 c_1 \circ\cdots\circ c_{m-1}|\circ c_m \quad , \quad k
\end{equation} 
If we refer to \textit{an optimal pattern into another one} (see \Prp{five}), then it means that we only use the allowed pattern obtained after removing the last part.
\subsection{Optimal forbidden patterns of $\Ee$ for four primary colors}
For four primary colors $a<b<c<d$, recall \eqref{orD} the total order on primary and secondary colors
\begin{equation}
 ab<ac<ad<a<bc<bd<b<cd<c<d
\end{equation}
and the set of special pairs $\Ccc = \{(ad,bc),(cd,ab)\}$.
\begin{theo}\label{four}
The optimal forbidden patterns are the following:
\begin{align}
 cd \rightarrow ab |&\twoheadrightarrow c,d \quad ,\quad k\geq 1\\
 ad \rightarrow bc |&\twoheadrightarrow a \quad \quad,\quad  k\geq 2\quad\cdot
\end{align} 
\end{theo}
\begin{proof}[Proof]
Let us consider the following diagram:
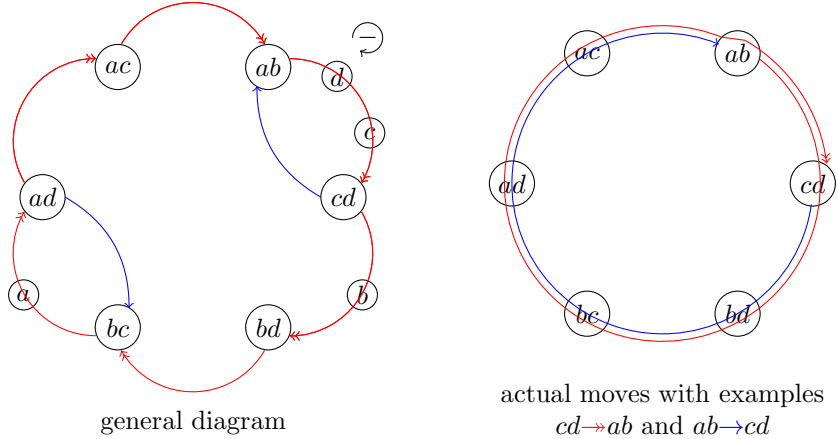
\begin{center}
\begin{tikzpicture}
 \draw [->] (45:3) arc (180:-120:0.2);
\draw (42.45:3.15) node {$-$};
 \foreach \x in {0,...,5}
 \draw (60*\x:2) circle (0.3);
  \draw (60:2) node {$ab$};
 \draw (120:2) node {$ac$};
 \draw (180:2) node {$ad$};
 \draw (240:2) node {$bc$};
 \draw (300:2) node {$bd$};
  \draw (210:2.6) circle (0.2) node {$a$};
  \draw (330:2.6) circle (0.2) node {$b$};
  \draw (20:2.5) circle (0.2) node {$c$};
   \draw (40:2.5) circle (0.2) node {$d$};
 \draw (0:2) node {$cd$};
 \foreach \x in {0,...,9}
 \draw [red, ->>] (60*\x-5:2.25) arc (60*\x+30:60*\x-90:1.1);
 \draw [blue, ->] (0:1.7) to[bend left] (60:1.7);
 \draw [blue, ->] (180:1.7) to[bend left] (240:1.7);
 \draw (-90:3) node {general diagram};
\end{tikzpicture}
$\quad\quad\quad$
\begin{tikzpicture}
 \foreach \x in {0,...,5}
 \draw (60*\x:2) circle (0.3);
 \draw (60:2) node {$ab$};
  \draw (120:2) node {$ac$};
 \draw (180:2) node {$ad$};
 \draw (240:2) node {$bc$};
 \draw (300:2) node {$bd$};
 \draw (0:2)node {$cd$};
 \draw [blue, ->] (-8:2) arc (-8:-292:2);
 \draw [red, ->>] (60:2.2) arc (60:7:2.2);
 \draw [red] (52:2.1) arc (52:-292:2.1);
 \draw [red] (68:2.1)--(60:2.2);
 \draw (-90:3) node[align=center] {actual moves with examples\\
 $cd \textcolor{red}{\twoheadrightarrow} ab$ and $ab \textcolor{blue}{\rightarrow} cd$};
 \end{tikzpicture}
\end{center}
We can see that the main nodes are the secondary colors, and we remark that a move $p \textcolor{red}{\twoheadrightarrow} q$ is indeed between 
$p$ and the color $q$ of the greatest secondary part smaller, in terms of $\triangleright$,
than a part with color $p$. Thus, any move $p \textcolor{red}{\twoheadrightarrow} q'$ with another secondary color $q'$ will be \textit{greater} than the move 
$p \textcolor{red}{\twoheadrightarrow} q$ represented in the first diagram.
As we notice on the second diagram, proceeding clockwise, we need more than one loop for a move $p \textcolor{red}{\twoheadrightarrow} q$, 
while a move $p\textcolor{blue}{\rightarrow} q$ requires less than one loop.
\\Since a forbidden pattern must necessarily begin with a sequence of secondary parts not well-ordered by $\triangleright$, we then have as the head of the pattern
either $cd \rightarrow ab$ or $ad\rightarrow bc$. 
\begin{itemize}
 \item Suppose that the pattern begins by $cd \rightarrow ab$. By \eqref{between}, if it ends by a primary part $k_{c_s}$, by setting $\nu_1+\nu_2=h_{cd}$  we then have   
 \[h_{ab}+1\succ k_{c_s}+s\succ h_{cd}\]
 so that $c_m\in \{c,d\}$. Another interpretation is that, in the diagram, the color $c_m$ is in the clockwise arc $(ab,cd)$, and it leads to the same result.
 Suppose now that $s\geq 3$, which means that the third part is secondary. Since the next move can be at least $ab\twoheadrightarrow cd$, we 
 then obtain that
 \[h_{cd}-2\succeq \nu_{5}+\nu_6 \Longrightarrow \overline{h_{cd}-2} \succeq \overline{\nu_{5}+\nu_6}\cdot\]
 This contradicts \eqref{loop}. Therefore, $s=2$ and, by \eqref{between}, we obtain the pattern $cd \rightarrow ab \twoheadrightarrow c,d$. It actually corresponds to the pattern 
 $(k+2)_{cd},(k+2)_{ab},k_{c,d}$. Here $k_{c,d}$ means $k_c$ or $k_d$. Since we must necessarily have that 
 \[\beta((k+2)_{ab})\not\succ k_{c,d}\]
 and a quick check according to the parity of $k$ shows that is always the case for $k\geq 1$.
 \item The same reasoning occurs when the pattern begins by $ad \rightarrow bc$. We obtain the pattern $ad \rightarrow bc \twoheadrightarrow a$ which corresponds to $(k+2)_{ad},(k+1)_{bc},k_a$. 
 We then look for $k$ such that 
 \[\beta((k+1)_{bc})\not\succ k_{a}\]
 and a quick check according to the parity of $k$ shows that is always the case for $k\geq 2$.
\end{itemize}
Note that we cannot have a optimal forbidden pattern consisting of three secondary parts, since whatever the head is, the third secondary part does not respect the relation \eqref{loop}.
\end{proof}
\Thm{four} and \Prp{pr7} imply that, for four primary colors, we do not have any shortcut. This is not the case for more than four primary colors, as we now see in the next subsection.
\bi
\subsection{Optimal forbidden patterns of $\Ee$ for more than four primary colors} We can restrict the study to five colors, as the set of colored partitions generated by five primary colors is embedded in any set of colored partitions generated by more than four primary colors. 
We then consider the set of primary colors $\C = \{a<b<c<d<e\}$.
The corresponding diagram with the primary equivalence classes for the secondary colors gives
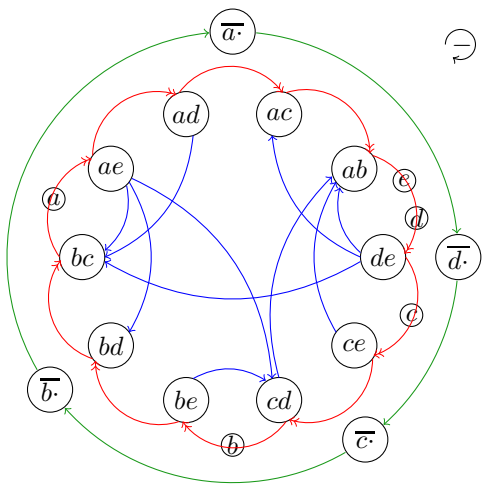
\begin{center}
 \begin{tikzpicture}
\draw [->] (45:4) arc (180:-120:0.2);
\draw (42.7:4.15) node {$-$};
 \foreach \x in {0,...,9}
 \draw (36*\x:2) circle (0.3);
 \draw (36:2) node {$ab$};
 \draw (72:2) node {$ac$};
 \draw (108:2) node {$ad$};
 \draw (144:2) node {$ae$};
 \draw (180:2) node {$bc$};
 \draw (0:2) node {$de$};
 \draw (-36:2) node {$ce$};
 \draw (-72:2) node {$cd$};
 \draw (-108:2) node {$be$};
 \draw (-144:2) node {$bd$};
 \draw (12:2.5) circle (0.15) node {$d$};
 \draw (24:2.5) circle (0.15) node {$e$};
 \draw (162:2.5) circle (0.15) node {$a$};
 \draw (270:2.5) circle (0.15) node {$b$};
 \draw (342:2.5) circle (0.15) node {$c$};
 \draw (90:3) circle (0.3) node {$\overline{a\cdot}$};
 \draw (-144:3) circle (0.3) node {$\overline{b\cdot}$};
 \draw (-54:3) circle (0.3) node {$\overline{c\cdot}$};
 \draw (0:3) circle (0.3) node {$\overline{d\cdot}$};
 
 \foreach \x in {0,...,9}
 \draw [red, ->>] (36*\x:2.3) arc (36*\x+36:36*\x-72:0.85);
 \draw [blue, ->] (2:1.7) to[bend left] (33:1.7);
 \draw [blue, ->] (0:1.7) to[bend left] (72:1.7);
 \draw [blue, ->] (-2:1.7) to[bend left] (-178:1.7);
 \draw [blue, ->] (-36:1.7) to[bend left] (36:1.7);
 \draw [blue, ->] (-69:1.72) to[bend left] (39:1.7);
 \draw [blue, ->] (108:1.7) to[bend left] (180:1.7);
 \draw [blue, ->] (146:1.7) to[bend left] (177:1.7);
 \draw [blue, ->] (-108:1.7) to[bend left] (-75:1.7);
 \draw [blue, ->] (142:1.7) to[bend left] (-72:1.7);
 \draw [blue, ->] (144:1.7) to[bend left] (-144:1.7);
 
 \draw [foge, ->] (-6:3) arc (-6:-48:3);
 \draw [foge, ->] (-60:3) arc (-60:-138:3);
 \draw [foge, ->] (-150:3) arc (-150:-264:3);
 \draw [foge, ->] (-276:3) arc (-276:-354:3);
\end{tikzpicture}
\end{center}
Let us first discuss the behaviour of the patterns with moves $\rightarrow p\rightarrow$. We can see in the diagram that this happens only if $p=cd$. Consider now the pattern
$$ae \rightarrow cd \rightarrow ab \twoheadrightarrow de \rightarrow bc \quad ,\quad k$$
which actually represents the pattern
\[(k+3)_{ae},(k+2)_{cd},(k+2)_{ab},k_{de},k_{bc}\,\cdot\]
We notice that this pattern is a shortcut. As we saw in \Prp{pr7}, the enumeration of the forbidden patterns then becomes intricate. We give the following lemma to restrict our study to some particular patterns without shortcut.
\begin{lem}\label{lem7}
For five primary colors, the patterns of secondary parts without the moves $\rightarrow cd \rightarrow$ do not contain any shortcut.
\end{lem}
The patterns without shortcut listed by the previous lemma are not exhaustive. In fact, we can have a pattern with moves $\rightarrow cd \rightarrow$ without shortcut, as we give in the following example.
\begin{ex}
The pattern $ae \rightarrow cd \rightarrow ab\,,\,k$ 
is not a shortcut and is even allowed for $k\neq 3$.
\end{ex}
The following theorem gives an exhaustive list of optimal forbidden patterns without moves $\rightarrow cd \rightarrow$. The notation $<g_1,\ldots,g_t>$ denotes the \textit{multiplicative} group generated by $g_1,\ldots,g_t$, and the notation $(pattern)$ means that the move $pattern$ is optional. 
\bi
\begin{theo}\label{five}
The optimal forbidden patterns with no move $\rightarrow p\rightarrow$ are the following:
\begin{align}
\textbf{\underline{head} :}\quad ad \rightarrow bc\nonumber\\
 ad \rightarrow bc &|\quad \twoheadrightarrow a\quad\quad,\quad  k\geq 2 \label{p1}\\
 \textbf{\underline{head} :}\quad be \rightarrow cd\nonumber\\
 be \rightarrow cd &|\quad \twoheadrightarrow b \quad\quad,\quad  k\geq 2 \label{p2}\\
 \textbf{\underline{head} :}\quad de \rightarrow ab\nonumber\\
 de \rightarrow ab &|\quad \twoheadrightarrow d,e \quad ,\quad  k\geq 1 \label{p3}\\
 \textbf{\underline{head} :}\quad de \rightarrow ac\nonumber\\
 de \rightarrow ac (\twoheadrightarrow ab) &|\quad\twoheadrightarrow  d,e \quad ,\quad  k\geq 1 \label{p4}\\
 \textbf{\underline{head} :}\quad ae \rightarrow bc\nonumber\\
 ae \rightarrow bc &|\quad \twoheadrightarrow a \quad \quad,\quad  k\geq 2 \label{p5}\\
 \textbf{\underline{head} :}\quad ae \rightarrow bd\nonumber\\
 ae \rightarrow bd (\twoheadrightarrow bc) &|\quad\twoheadrightarrow a \quad\quad ,\quad  k\geq 2 \label{p6}\\
 \textbf{\underline{head} :}\quad ae \rightarrow cd\nonumber\\
 ae \rightarrow cd  &|\quad \twoheadrightarrow b \quad\quad ,\quad  k\geq 2 \label{p20}\\
 ae \rightarrow cd (\pat) &|\quad \twoheadrightarrow a \quad\quad ,\quad  k\geq 2 \label{p7}\\
 \text{where  } \pat \in <\twoheadrightarrow\eqref{p2}>&& \nonumber\\
 \eqref{p7} (\twoheadrightarrow be) (\twoheadrightarrow bd) (\twoheadrightarrow bc)&|\quad\twoheadrightarrow a \quad\quad ,\quad  k\geq 2 \label{p8}\\
 \textbf{\underline{head} :}\quad de \rightarrow bc\nonumber\\
 de \rightarrow bc &|\quad\twoheadrightarrow a \quad\quad ,\quad  k\geq 2 \label{p21}\\
 de \rightarrow bc \quad (\pat)&|\quad\twoheadrightarrow e \quad\quad ,\quad  k\geq 1  \label{p9}\\
 \text{where  } \pat \in <\twoheadrightarrow\eqref{p8}, \twoheadrightarrow \eqref{p6},\twoheadrightarrow \eqref{p5},(\twoheadrightarrow ae)\twoheadrightarrow \eqref{p1}>&& \nonumber\\
 \eqref{p9}(\twoheadrightarrow ae)(\twoheadrightarrow ad)(\twoheadrightarrow ac)(\twoheadrightarrow ab)&|\quad\twoheadrightarrow e \quad\quad ,\quad  k\geq 1  \label{p10}\\
 \eqref{p10}&|\quad\twoheadrightarrow d \quad\quad ,\quad  k\geq 2 \label{p40} \\
 \eqref{p10}&|\quad\twoheadrightarrow d \quad\quad ,\quad  1  \label{p41}\\
 \text{with } \eqref{p10} \text{ not ending by }ae,be&& \nonumber\\
 \eqref{p9}\twoheadrightarrow \eqref{p7}&|\quad\twoheadrightarrow  be,bd\quad ,\quad  2\label{p12}\\
 \eqref{p9}\twoheadrightarrow \eqref{p8}&|\quad\twoheadrightarrow  ae\quad\quad ,\quad  2 \label{p13}\\
 \eqref{p13}&|\quad\twoheadrightarrow ad \quad\quad ,\quad  2  \label{p31}\\
 \text{with } \eqref{p13} \text{ not ending by }be&& \nonumber\\
 \textbf{\underline{head} :}\quad cd,ce \rightarrow ab\nonumber\\
  cd,ce \rightarrow ab &|\quad \twoheadrightarrow d,e \quad ,\quad  k\geq 1 \label{p22} \\
 cd,ce \rightarrow ab (\pat)&|\quad \twoheadrightarrow c \quad\quad ,\quad  k\geq 2  \label{p14}\\
 \text{where  } \pat \in <\twoheadrightarrow\eqref{p3}, \twoheadrightarrow \eqref{p4},\twoheadrightarrow \eqref{p10}>&& \nonumber\\
\eqref{p14} \twoheadrightarrow de&|\quad \twoheadrightarrow c \quad\quad ,\quad  k\geq 2  \label{p15}\\
 \eqref{p14}&|\quad\twoheadrightarrow c \quad\quad ,\quad  1  \\
 \text{with } \eqref{p14} \text{ ending by }ac,ab,bc&& \nonumber\\
 \eqref{p14} \twoheadrightarrow \eqref{p12}\twoheadrightarrow be &|\quad\rightarrow cd \quad\quad ,\quad  3  \\
  \eqref{p14} \twoheadrightarrow \eqref{p13}\twoheadrightarrow ae &|\quad\rightarrow cd \quad\quad ,\quad  3  \\
  \eqref{p14}\twoheadrightarrow \eqref{p9} (\twoheadrightarrow ae)\twoheadrightarrow ad &|\quad\twoheadrightarrow ac \quad\quad ,\quad  2  \\
  \eqref{p14}\twoheadrightarrow \eqref{p9} &|\quad\twoheadrightarrow ac \quad\quad ,\quad  2  \\
 \text{with } \eqref{p9} \text{ ending by }bc&& \nonumber
 \end{align} 
\end{theo}
\begin{proof}[Proof of \Thm{five}]
We recall that the optimal forddiden patterns 
$$\nu= \nu_1+\nu_2\,\,\, \nt\nu_3+\nu_4\,\gg\,\cdots\,\, \gg\,\, \nu_{2s+1}$$
with no shortcut have the form described in \eqref{opt}:
\[
 c_1 \circ\cdots\circ c_{s}|\circ c_{s+1} \quad , \quad k \quad\cdot
\]
The part $\nu_{2i-1}+\nu_{2i}$ has the secondary color $c_i$ for all $i\in[1,s]$, and the primary part $\nu_{2s+1}$ has the color $c_{s+1}$.
\begin{itemize}
 \item[\textbf{Rule 1 :}] \textbf{For all $i\in [2,s]$, $c_{s+1}$ belongs to the clockwise
 arc $(\overline{c_i},\overline{c_1})$.}
 In fact, by \eqref{arc1}, we have that
 $$\nu_{2s+1} +s-i+2\succ \overline{\nu_1+\nu_2 -i+2}\succeq \overline{\nu_{2i-1}+\nu_{2i}} \succ \nu_{2s+1} +s-i+1\,,$$
 so that by starting a clockwise loop in the diagram from $\overline{c_i}$, we respectively meet $c_{s+1},\overline{c_1}$ and $\overline{c_i}$. 
 \item[\textbf{Rule 2 :}] \textbf{If we have  a move $c_i\twoheadrightarrow c_{i+1}$, then $c_{i+1}$ strictly belongs to the clockwise
 arc $(c_i,c_{s+1})$.}
 In fact, we have by the primary equivalence definition and \eqref{arc1} that 
 \[\nu_{2s+1} +s+2-i \succ\nu_{2i-1}+\nu_{2i}\succ \nu_{2s+1} +s+1-i \succ \nu_{2i+1}+\nu_{2i+2}\succ \nu_{2s+1} +s-i \]
 and the move $c_i\twoheadrightarrow c_{i+1}$ implies that 
 \[\nu_{2i-1}+\nu_{2i}\triangleright \nu_{2i+1}+\nu_{2i+2} \Longleftrightarrow \nu_{2i-1}+\nu_{2i}-1\succ \nu_{2i+1}+\nu_{2i+2}\,\cdot\]
We thus obtain the following inequality
 \[\nu_{2s+1} +s+1-i  \succ \nu_{2i-1}+\nu_{2i}-1\succ \nu_{2i+1}+\nu_{2i+2}\succ \nu_{2s+1} +s-i\,\cdot\]
\end{itemize}
With these two rules, we can retrieve all the optimal forddiden patterns. In our construction, we will see that our moves are indeed mimimal for $\gg$. This means that, in the case where $(c_i,c_{i+1})\in \Ccc$, we necessarily make the move $c_i \rightarrow c_{i+1}$. By \Lem{lem6}, with the minimality of the consecutive size differences, once the part $\nu_{2s+1}$ crosses the parts $\nu_{2s-1}+\nu_{2s}$, it then crosses all the parts up to $\nu_1+\nu_2$.
Therefore, the choice of the size $k$ is such that the part $k_{c_{s+1}}$ crosses the last secondary part $(k+1+\chi(c_s\leq c_{s+1}))_{c_s}$. We thus have
\begin{equation}
 k_{c_{s+1}}\succeq \beta((k+1+\chi(c_s\leq c_{s+1}))_{c_s})\,\cdot
\end{equation}
We then proceed as follows.
\begin{enumerate}
 \item We select a head $c_1\rightarrow c_2$, and $c_{s+1}$ a primary color in the clockwise arc $(c_2,c_1)$. The best way is to begin with those with the shortest arc.
 \item The next move must necessarily be of the form $c_2\twoheadrightarrow c_3$.
 \begin{enumerate}
 \item With Rule 2, the patterns \eqref{p1},\eqref{p2},\eqref{p3} and \eqref{p5} follow immediately. In fact, in these cases, the only primary colors in the arc $(c_1,c_2)$ directly follow $c_2$ in the clockwise
 sense before all the secondary colors.
 \begin{center}
 \begin{tikzpicture}[scale=0.5, every node/.style={scale=0.5}]
\draw [->] (45:4) arc (180:-120:0.2);
\draw (42.7:4.15) node {$-$};
 \foreach \x in {0,...,9}
 \draw (36*\x:2) circle (0.3);
 \draw (36:2) node {$ab$};
 \draw (72:2) node {$ac$};
 \draw (108:2) node {$ad$};
 \draw (144:2) node {$ae$};
 \draw (180:2) node {$bc$};
 \draw (0:2) node {$de$};
 \draw (-36:2) node {$ce$};
 \draw (-72:2) node {$cd$};
 \draw (-108:2) node {$be$};
 \draw (-144:2) node {$bd$};
 \draw (12:2.5) circle (0.15) node {$d$};
 \draw (24:2.5) circle (0.15) node {$e$};
 \draw (162:2.5) circle (0.15) node {$a$};
 \draw (270:2.5) circle (0.15) node {$b$};
 \draw (342:2.5) circle (0.15) node {$c$};
 \draw [blue, ->] (2:1.7) to[bend left] (33:1.7);
 \draw [blue, ->] (108:1.7) to[bend left] (180:1.7);
 \draw [blue, ->] (146:1.7) to[bend left] (177:1.7);
 \draw [blue, ->] (-108:1.7) to[bend left] (-75:1.7);
 \foreach \x in {1,5,8}
 \draw [red, ->>] (36*\x:2.3) arc (36*\x+36:36*\x-72:0.85);
\end{tikzpicture}
\end{center}
 \item We also obtain the patterns \eqref{p20},\eqref{p21}, and \eqref{p22} since the chosen primary color is directly after $c_2$.
 \begin{center}
 \begin{tikzpicture}[scale=0.5, every node/.style={scale=0.5}]
\draw [->] (45:4) arc (180:-120:0.2);
\draw (42.7:4.15) node {$-$};
 \foreach \x in {0,...,9}
 \draw (36*\x:2) circle (0.3);
 \draw (36:2) node {$ab$};
 \draw (72:2) node {$ac$};
 \draw (108:2) node {$ad$};
 \draw (144:2) node {$ae$};
 \draw (180:2) node {$bc$};
 \draw (0:2) node {$de$};
 \draw (-36:2) node {$ce$};
 \draw (-72:2) node {$cd$};
 \draw (-108:2) node {$be$};
 \draw (-144:2) node {$bd$};
 \draw (12:2.5) circle (0.15) node {$d$};
 \draw (24:2.5) circle (0.15) node {$e$};
 \draw (162:2.5) circle (0.15) node {$a$};
 \draw (270:2.5) circle (0.15) node {$b$};
 \draw (342:2.5) circle (0.15) node {$c$};
 
 \foreach \x in {1,5,8}
 \draw [red, ->>] (36*\x:2.3) arc (36*\x+36:36*\x-72:0.85);
 \draw [blue, ->] (-2:1.7) to[bend left] (-178:1.7);
 \draw [blue, ->] (-36:1.7) to[bend left] (36:1.7);
 \draw [blue, ->] (-69:1.72) to[bend left] (39:1.7);
 \draw [blue, ->] (142:1.7) to[bend left] (-72:1.7);
\end{tikzpicture}
\end{center}
 \item In the case \eqref{p4} and \eqref{p6}, there is only one secondary color in the arc which occurs before the chosen primary color, and we can 
 see that from this color we only have moves of the form $\twoheadrightarrow$. The only possibility if we choose $c_3$ to be this secondary color will be then to directly reach the primary color at $c_4$.
 We can also decide to choose $c_3$ as the primary color. We recall that
 \[c_1\rightarrow c_2 (\twoheadrightarrow c_3)| \twoheadrightarrow c_4\]
 means that the choice of the secondary color in between $c_2$ and the primary color $c_4$ is optional.
 
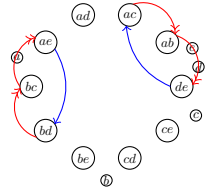
\begin{center}
 \begin{tikzpicture}[scale=0.5, every node/.style={scale=0.5}]
\draw [->] (45:4) arc (180:-120:0.2);
\draw (42.7:4.15) node {$-$};
 \foreach \x in {0,...,9}
 \draw (36*\x:2) circle (0.3);
 \draw (36:2) node {$ab$};
 \draw (72:2) node {$ac$};
 \draw (108:2) node {$ad$};
 \draw (144:2) node {$ae$};
 \draw (180:2) node {$bc$};
 \draw (0:2) node {$de$};
 \draw (-36:2) node {$ce$};
 \draw (-72:2) node {$cd$};
 \draw (-108:2) node {$be$};
 \draw (-144:2) node {$bd$};
 \draw (12:2.5) circle (0.15) node {$d$};
 \draw (24:2.5) circle (0.15) node {$e$};
 \draw (162:2.5) circle (0.15) node {$a$};
 \draw (270:2.5) circle (0.15) node {$b$};
 \draw (342:2.5) circle (0.15) node {$c$};
 
 \foreach \x in {1,2,5,6}
 \draw [red, ->>] (36*\x:2.3) arc (36*\x+36:36*\x-72:0.85);
 \draw [blue, ->] (144:1.7) to[bend left] (-144:1.7);
 \draw [blue, ->] (0:1.7) to[bend left] (72:1.7);
\end{tikzpicture}
\end{center}
 \end{enumerate} 
 For all these cases, one can check that it is not possible to build from them some forbidden pattern with only secondary parts. 
\item The remaining case is where $c_3$ is between in the arc $(c_2,c_{s+1})$ and such that we can have a move $c_3\rightarrow c_4$.
We then use the following property of our optimal forddiden pattern  due to \eqref{arc1}: 
\textit{when we do $m$ moves from the first color to another secondary color, in the diagram, we do around the first color fewer than $m$ but at least $m-1$ primary loops.}
This means that, by taking the allowed pattern resulting from the removal of the last part in an optimal forbidden pattern beginning by $c_3\rightarrow c_4$, we will always satisfy \eqref{arc1}.
For this reason, we wisely begin with $c_1\rightarrow c_2= ae \rightarrow cd$ and $c_{s+1}=a$.
\begin{enumerate}
 \item For $c_1\rightarrow c_2= ae \rightarrow cd$ and $c_{s+1}=a$. 
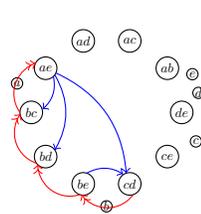
 \begin{center}
 \begin{tikzpicture}[scale=0.5, every node/.style={scale=0.5}]
\draw [->] (45:4) arc (180:-120:0.2);
\draw (42.7:4.15) node {$-$};
 \foreach \x in {0,...,9}
 \draw (36*\x:2) circle (0.3);
 \draw (36:2) node {$ab$};
 \draw (72:2) node {$ac$};
 \draw (108:2) node {$ad$};
 \draw (144:2) node {$ae$};
 \draw (180:2) node {$bc$};
 \draw (0:2) node {$de$};
 \draw (-36:2) node {$ce$};
 \draw (-72:2) node {$cd$};
 \draw (-108:2) node {$be$};
 \draw (-144:2) node {$bd$};
 \draw (12:2.5) circle (0.15) node {$d$};
 \draw (24:2.5) circle (0.15) node {$e$};
 \draw (162:2.5) circle (0.15) node {$a$};
 \draw (270:2.5) circle (0.15) node {$b$};
 \draw (342:2.5) circle (0.15) node {$c$};
 \foreach \x in {5,...,8}
 \draw [red, ->>] (36*\x:2.3) arc (36*\x+36:36*\x-72:0.85);
 \draw [blue, ->] (146:1.7) to[bend left] (177:1.7);
 \draw [blue, ->] (-108:1.7) to[bend left] (-75:1.7);
 \draw [blue, ->] (142:1.7) to[bend left] (-72:1.7);
 \draw [blue, ->] (144:1.7) to[bend left] (-144:1.7);
\end{tikzpicture}
\end{center}
If $c_3 \neq c_{s+1}=a$, by both rules, we have that $c_3 \in \{be,bd,bc\}$. As soon as $c_3 \neq be$, 
 we obtain by the second rule that the pattern is
 \[ae \rightarrow cd\twoheadrightarrow bc |\twoheadrightarrow a \quad \text{ or } \quad ae \rightarrow cd\twoheadrightarrow bd (\twoheadrightarrow bc) |\twoheadrightarrow a\,\cdot\]
 If $c_3 = be$, then we can iterate the pattern \eqref{p2} (which is $be\rightarrow cd $) as many times as we want. By doing this, we do as many loops as the number of moves, which is twice the number of
 iterations. However, once we move out from this iteration, we can only move to $a$ by optionally passing by $be,bd,bc$ through $\twoheadrightarrow$. 
 In fact, anytime we reach $cd$, we cannot make a move $cd\rightarrow$, so that by the second rule, we need to move back to either $be,bd,bc$ or $a$ using $\twoheadrightarrow$. We then obtain the patterns \eqref{p7} and \eqref{p8}.
 Note that for these patterns, we stay in the arc $(cd,a)$, and the passage from $ae=c_1$ to $c_s$ requires more than $s-1$ primary loops, so that the pattern
 \[ae \cdots c_s \twoheadrightarrow ae\]
 requires $s+1$ primary loops. We also observe that apart from $c_1=ae$ and $c_{s+1}$, all colors $c_i$ belong to $\{cd,be,bd,bc\}$, so that their upper halves can never be a primary part with color $a$ and
 we do not have any optimal forbidden patterns with only secondary parts coming from a forbidden pattern of that form.
 \item For $c_1\rightarrow c_2= de \rightarrow bc$ and $c_{s+1}=d,e$. 
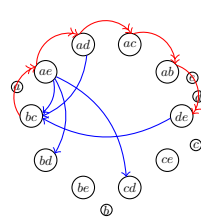
 \begin{center}
 \begin{tikzpicture}[scale=0.5, every node/.style={scale=0.5}]
\draw [->] (45:4) arc (180:-120:0.2);
\draw (42.7:4.15) node {$-$};
 \foreach \x in {0,...,9}
 \draw (36*\x:2) circle (0.3);
 \draw (36:2) node {$ab$};
 \draw (72:2) node {$ac$};
 \draw (108:2) node {$ad$};
 \draw (144:2) node {$ae$};
 \draw (180:2) node {$bc$};
 \draw (0:2) node {$de$};
 \draw (-36:2) node {$ce$};
 \draw (-72:2) node {$cd$};
 \draw (-108:2) node {$be$};
 \draw (-144:2) node {$bd$};
 \draw (12:2.5) circle (0.15) node {$d$};
 \draw (24:2.5) circle (0.15) node {$e$};
 \draw (162:2.5) circle (0.15) node {$a$};
 \draw (270:2.5) circle (0.15) node {$b$};
 \draw (342:2.5) circle (0.15) node {$c$};
 \foreach \x in {1,...,5}
 \draw [red, ->>] (36*\x:2.3) arc (36*\x+36:36*\x-72:0.85);
 \draw [blue, ->] (-2:1.7) to[bend left] (-178:1.7);
 \draw [blue, ->] (108:1.7) to[bend left] (180:1.7);
 \draw [blue, ->] (146:1.7) to[bend left] (177:1.7);
 \draw [blue, ->] (142:1.7) to[bend left] (-72:1.7);
 \draw [blue, ->] (144:1.7) to[bend left] (-144:1.7);
\end{tikzpicture}
\end{center}
 We use the same reasoning to show that the only moves that can leave the arc $(bc,a)$ are \eqref{p1}, \eqref{p5},\eqref{p6} and \eqref{p8}. For \eqref{p1} (the move $ad\rightarrow bc$), in order 
 to make as many loops as the number of moves, we can optionally add a move $\twoheadrightarrow ae \twoheadrightarrow$ before reaching $ad$. This is why we can compose a pattern using the patterns 
  \eqref{p5},\eqref{p6} and \eqref{p8} and $ae \twoheadrightarrow \eqref{p1}$, and we obtain \eqref{p9}. In this composition, we can remark that we do not make a move $cd\rightarrow$. In fact, the only way to reach $cd$ is to do a move
  \eqref{p8}, but in this move $cd$ can only be reached after the move $ae\rightarrow cd$, so that we cannot do $cd\rightarrow$.
  \\Once we move out of this composition, we can only reach the primary color $d,e$ by optionally passing by  the primary equivalent class  $\overline{a.}$, which consists of the secondary colors $ae,ad,ac,ab$. In addition, these moves have the form $\twoheadrightarrow$. We then obtain \eqref{p10},
  \eqref{p40} and \eqref{p41}.
  Note that for these patterns, the secondary colors stay in the arc $(cd,d)$, and the passage from $de=c_1$ to $c_s$ requires more than $s-1$ primary loops, so that the pattern
 \[de \cdots c_s \twoheadrightarrow de\]
 requires $s+1$ primary loops. To obtain the forbidden patterns with only secondary colors, we just need to choose those which correspond to the forbidden patterns ending by a primary color and such that 
 the upper half of the last part corresponds to the primary color and is at least equal than the lower  half of the previous secondary part. We then have the patterns \eqref{p12},\eqref{p13} and \eqref{p31}.
 \item For $c_1\rightarrow c_2= cd,ce \rightarrow bc$ and $c_{s+1}=c$. 
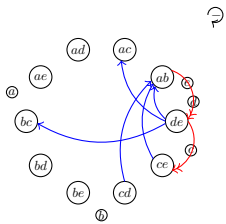
 \begin{center}
 \begin{tikzpicture}[scale=0.5, every node/.style={scale=0.5}]
\draw [->] (45:4) arc (180:-120:0.2);
\draw (42.7:4.15) node {$-$};
 \foreach \x in {0,...,9}
 \draw (36*\x:2) circle (0.3);
 \draw (36:2) node {$ab$};
 \draw (72:2) node {$ac$};
 \draw (108:2) node {$ad$};
 \draw (144:2) node {$ae$};
 \draw (180:2) node {$bc$};
 \draw (0:2) node {$de$};
 \draw (-36:2) node {$ce$};
 \draw (-72:2) node {$cd$};
 \draw (-108:2) node {$be$};
 \draw (-144:2) node {$bd$};
 \draw (12:2.5) circle (0.15) node {$d$};
 \draw (24:2.5) circle (0.15) node {$e$};
 \draw (162:2.5) circle (0.15) node {$a$};
 \draw (270:2.5) circle (0.15) node {$b$};
 \draw (342:2.5) circle (0.15) node {$c$};
 \foreach \x in {0,1}
 \draw [red, ->>] (36*\x:2.3) arc (36*\x+36:36*\x-72:0.85);
 \draw [blue, ->] (2:1.7) to[bend left] (33:1.7);
 \draw [blue, ->] (0:1.7) to[bend left] (72:1.7);
 \draw [blue, ->] (-2:1.7) to[bend left] (-178:1.7);
 \draw [blue, ->] (-36:1.7) to[bend left] (36:1.7);
 \draw [blue, ->] (-69:1.72) to[bend left] (39:1.7);
\end{tikzpicture}
\end{center}
 We use the same reasoning to show that the only moves that can leave the arc $(ab,c)$ are \eqref{p10}, \eqref{p4},\eqref{p3}.
 As before, in the composition of these moves, we remark that we do not make a move $cd\rightarrow$ and the secondary colors stay in the clockwise arc $(cd,c)$.
 Once we do not make these moves, we can only go to $c$  by optionally passing by $de$ through $\twoheadrightarrow$. 
 For these patterns, the passage from $de=c_1$ to $c_s$ requires more than $s-1$ primary loops, so that the pattern
 \[cd,ce \cdots c_s \twoheadrightarrow ce,cd\]
 requires $s+1$ primary loops.
 We obtain the optimal forbidden patterns consisting of only secondary parts, always by choosing those corresponding to optimal forbidden patterns ending primary colors and such that 
 the upper half of the last part corresponds to the primary color and is at least equal than the lower  half of the previous secondary part.
 \end{enumerate}
\end{enumerate} 
\end{proof}
To conclude, we see that for more than four colors, there exist some shortcuts. However, even for five colors, the set of optimal forbidden patterns without shorcut is infinite, as a consequence of \Thm{five}, since some patterns use as many iterations of others.  The enumeration of the forbidden patterns then becomes intricate for more than four primary colors. 
\section{Bressoud's algorithm, Motzkin paths and oriented rooted forests}\label{sct7}
In this section, we relate the partitions in $\E$ to oriented rooted forests, and give a new potential approach to deal with the enumeration of the forbidden patterns.
\bi    
Let us take a partition $\nu\in \E$ and write it as 
\begin{equation}
\nu = (\nu_1,\cdots,\nu_{p+2s})\,,
\end{equation}
where as before, $p$ is the number of primary parts and $s$ is the number of secondary parts.
We recall that the set $J$ is the set of indices that correspond to the primary parts, and $I$ corresponds to the upper halves, so that $I+1$ 
is associated  to the lower halves.
\\We observe that the sequence $\la = \Psi(\nu)$ has also $p+2s$ primary parts. We then have $\la = \la_1,\ldots,\la_{p+s}$. For any $x\in [1,p+2s]$, we set $\theta_{x}$ to be the index in $\la$ of the primary part that comes from $\nu_x$. 
\begin{ex}
As an example, we apply $\Phi$ on the partition $\la=(12_a,7_b,6_d,6_c,5_a,4_d,4_c,4_b,4_a,3_c,1_d,1_c,1_b,1_a)$ and take $\nu = \Phi(\la)$:
\begin{small}
\begin{equation}\label{example2}
\begin{array}{ccccccccccccccccccc}
\begin{matrix}
12_a\\
\underline{7_b}\\
\underline{6_d}\\
6_c\\
5_a\\
4_d\\
4_c\\
4_b\\
4_a\\
3_c\\
1_d\\
1_c\\
1_b\\
1_a
\end{matrix} &\rightarrowtail&
\begin{matrix}
\mathbf{12_a}\\
\mathbf{13_{bd}}\\
6_c\\
5_a\\
4_d\\
4_c\\
4_b\\
4_a\\
3_c\\
1_d\\
1_c\\
1_b\\
1_a
\end{matrix}&
\looparrowright&
\begin{matrix}
14_{bd}\\
11_a\\
6_c\\
\underline{5_a}\\
\underline{4_d}\\
4_c\\
4_b\\
4_a\\
3_c\\
1_d\\
1_c\\
1_b\\
1_a
\end{matrix}
&
\rightarrowtail&
\begin{matrix}
14_{bd}\\
11_a\\
\mathbf{6_c}\\
\mathbf{9_{ad}}\\
4_c\\
4_b\\
4_a\\
3_c\\
1_d\\
1_c\\
1_b\\
1_a
\end{matrix}
&
\looparrowright&
\begin{matrix}
14_{bd}\\
11_a\\
10_{ad}\\
5_c\\
\underline{4_c}\\
\underline{4_b}\\
4_a\\
3_c\\
1_d\\
1_c\\
1_b\\
1_a
\end{matrix}
&
\rightarrowtail&
\begin{matrix}
14_{bd}\\
11_a\\
10_{ad}\\
\mathbf{5_c}\\
\mathbf{8_{bc}}\\
4_a\\
3_c\\
1_d\\
1_c\\
1_b\\
1_a
\end{matrix}
&
\looparrowright&
\begin{matrix}
14_{bd}\\
11_a\\
10_{ad}\\
9_{bc}\\
\underline{4_c}\\
\underline{4_a}\\
3_c\\
1_d\\
1_c\\
1_b\\
1_a
\end{matrix}
&
\rightarrowtail&
\begin{matrix}
14_{bd}\\
11_a\\
10_{ad}\\
9_{bc}\\
8_{ac}\\
3_c\\
\underline{1_d}\\
\underline{1_c}\\
1_b\\
1_a
\end{matrix}
&
\rightarrowtail&
\begin{matrix}
14_{bd}\\
11_a\\
10_{ad}\\
9_{bc}\\
8_{ac}\\
3_c\\
2_{cd}\\
\underline{1_b}\\
\underline{1_a}
\end{matrix}
&
\rightarrowtail&
\begin{matrix}
14_{bd}\\
11_a\\
10_{ad}\\
9_{bc}\\
8_{ac}\\
3_c\\
2_{cd}\\
2_{ab}
\end{matrix}
\end{array}\,\cdot
\end{equation}
\end{small}
We retrieve the partition $\nu$ of \Expl{ex51}. By considering the occurrences of the primary parts, we obtain the following diagram:
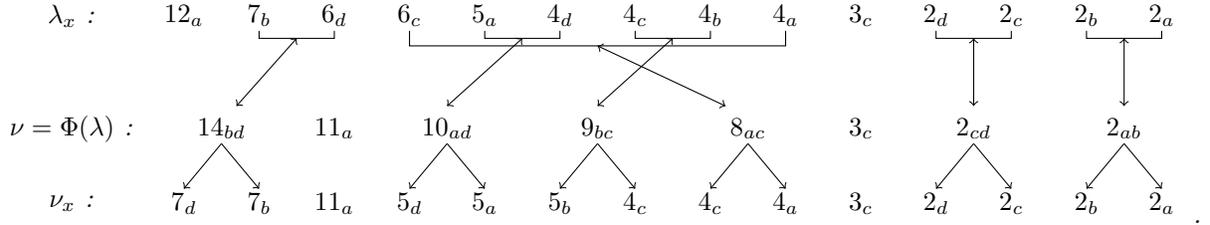
\begin{center}
\begin{tikzpicture}
\draw(-1.5,2) node{$\la_x$ :};
 \draw (0,2) node {$12_a$};
 \draw (1,2) node {$7_b$};
 \draw (2,2) node {$6_d$};
 \draw (3,2) node {$6_c$};
 \draw (4,2) node {$5_a$};
 \draw (5,2) node {$4_d$};
 \draw (6,2) node {$4_c$};
 \draw (7,2) node {$4_b$};
 \draw (8,2) node {$4_a$};
 \draw (9,2) node {$3_c$};
 \draw (10,2) node {$2_d$};
 \draw (11,2) node {$2_c$};
 \draw (12,2) node {$2_b$};
 \draw (13,2) node {$2_a$};
 \draw (1,1.8)--(1,1.7)--(2,1.7)--(2,1.8); \draw[<->] (1.5,1.7)--(0.7,0.8);
 \draw (3,1.8)--(3,1.6)--(8,1.6)--(8,1.8); \draw[<->] (5.5,1.6)--(7.2,0.8);
 \draw (4,1.8)--(4,1.7)--(5,1.7)--(5,1.8); \draw[<->] (4.5,1.7)--(3.5,0.8);
 \draw (6,1.8)--(6,1.7)--(7,1.7)--(7,1.8); \draw[<->] (6.5,1.7)--(5.5,0.8);
 \draw (10,1.8)--(10,1.7)--(11,1.7)--(11,1.8); \draw[<->] (10.5,1.7)--(10.5,0.8);
 \draw (12,1.8)--(12,1.7)--(13,1.7)--(13,1.8); \draw[<->] (12.5,1.7)--(12.5,0.8);
 \draw(-1.5,0.5) node{$\nu=\Phi(\la)$ :};
 \draw (0.5,0.5) node {$14_{bd}$}; \draw [->](0.5,0.3)--(0,-0.3);\draw [->](0.5,0.3)--(1,-0.3);
 \draw (2,0.5) node {$11_a$};
 \draw (3.5,0.5) node {$10_{ad}$}; \draw [->](3.5,0.3)--(3,-0.3);\draw [->](3.5,0.3)--(4,-0.3);
 \draw (5.5,0.5) node {$9_{bc}$}; \draw [->](5.5,0.3)--(5,-0.3);\draw [->](5.5,0.3)--(6,-0.3);
 \draw (7.5,0.5) node {$8_{ac}$}; \draw [->](7.5,0.3)--(7,-0.3);\draw [->](7.5,0.3)--(8,-0.3);
 \draw (9,0.5) node {$3_c$};
 \draw (10.5,0.5) node {$2_{cd}$}; \draw [->](10.5,0.3)--(10,-0.3);\draw [->](10.5,0.3)--(11,-0.3);
 \draw (12.5,0.5) node {$2_{ab}$}; \draw [->](12.5,0.3)--(12,-0.3);\draw [->](12.5,0.3)--(13,-0.3);
 \draw(-1.5,-0.5) node{$\nu_x$ :};
 \draw (0,-0.5) node {$7_d$};
 \draw (1,-0.5) node {$7_b$};
 \draw (2,-0.5) node {$11_a$};
 \draw (3,-0.5) node {$5_d$};
 \draw (4,-0.5) node {$5_a$};
 \draw (5,-0.5) node {$5_b$};
 \draw (6,-0.5) node {$4_c$};
 \draw (7,-0.5) node {$4_c$};
 \draw (8,-0.5) node {$4_a$};
 \draw (9,-0.5) node {$3_c$};
 \draw (10,-0.5) node {$2_d$};
 \draw (11,-0.5) node {$2_c$};
 \draw (12,-0.5) node {$2_b$};
 \draw (13,-0.5) node {$2_a$};

\end{tikzpicture} .
\end{center}
We recall that 
\begin{align*}
 (p,s)=(2,6),\quad J=\{3,10\}, \quad I = \{1,4,6,8,11,13\},\quad I+1 = \{2,5,7,9,12,14\}
\end{align*}
and we have
\begin{equation}
\begin{array}{|c|c|c|c|c|c|c|c|c|c|c|c|c|c|c|}
\hline
 x&1&2&3&4&5&6&7&8&9&10&11&12&13&14\\
 \hline
 \theta_x&2&3&1&5&6&7&8&4&9&10&11&12&13&14\\
 \hline
\end{array} \,\cdot            
\end{equation}
We also compute $\Br$ for $\nu=\Phi(\la)$ and we obtain
\begin{equation}
\begin{array}{|c|c|c|c|c|c|c|}
\hline
 i&1&4&6&8&11&13\\
 \hline
 \Br(i)&3&8&8&8&11&13\\
 \hline
\end{array} \,\cdot            
\end{equation}
\end{ex}
The most important results of this part are the following.
\bi 
\begin{prop}[\textbf{Motzkin path behaviour of the final positions}]\label{pr3}
For any $(i,i',j,j')\in I ^2\times J^2$, we have the following relations:
\begin{align}
&\text{If }i<i',\text{ then either}\quad \theta_{i}<\theta_{i+1}<\theta_{i'}<\theta_{i'+1}\quad\text{or}\quad \theta_{i'}<\theta_{i}<\theta_{i+1}<\theta_{i'+1}\,\cdot\label{ii}\\
&\text{If }j<j',\text{ then}\quad \theta_j<\theta_{j'}\,\cdot\label{jj}\\
&i+1\leq\theta_{i+1}\quad\text{and}\quad \theta_{j}\leq j\,\cdot \label{i1}\\
&\text{Either}\quad\theta_{j}<\theta_{i}\quad\text{or}\quad \theta_{i+1}<\theta_{j}\,\cdot\label{ij}
\end{align}
\end{prop}
\bi
\begin{prop}[\textbf{Bridge according to the final positions}]\label{pr4}
For any $i\in I$, we have the following: 
\begin{itemize}
 \item If there exists $i<j\in J$ such that $\theta_j<\theta_i$, then 
 \begin{equation}
 \Br(i)=\min\{j\in J: j>i\text{ and }\theta_j<\theta_i\}\,\cdot
 \end{equation}
 \item Otherwise, 
 \begin{equation}
 \Br(i)=\max\{ i'\in I: i'\geq  i \text{ and }\theta_{i'}\leq \theta_i\}\,\cdot
 \end{equation}
\end{itemize}
\end{prop}
\bi
\begin{rem}
We indeed have by \Prp{pr3} for all $i\in I$ that
$$
\theta_{i+1}-(i+1) = |\{u\in I\sqcup J: u>i\text{ and } \theta_u<\theta_i\}|\,,
$$
and \Prp{pr4} gives the following equivalence:
$$
\Br(i)=i \,\Longleftrightarrow\, \theta_{i+1}=i+1\,\cdot
$$
\end{rem}
\bi
Let us set $I=\{i_1<\cdots<i_s\}$ and $J^+ = J \sqcup\{0,p+2s+1\} = \{j_0<j_1<\cdots<j_p<j_{p+1}\}$ and $(\theta_0,\theta_{p+2s+1}) = (0,p+2s+1)$. Then, by \eqref{jj} and \eqref{ij}  of \Prp{pr3}, for any consecutive 
$j,j'\in J^+$, there exists a unique $ V\subset \{1,\ldots,s\}$ such that 
$$
 \{\theta_{j}+1,\ldots,\theta_{j'}-1\} = \{\theta_x: x\in \{i_v,i_{v}+1: v\in V\}\}\,\cdot
$$
This means that the final positions between those of consecutive primary parts consist of those of the upper and lower halves of some secondary parts.  
By \eqref{ii}, we can check that those secondary parts are consecutive, and $V$ is indeed an interval. Since the positions $\theta_{i+1}$ form an increasing sequence, we then have a unique decomposition 
$$
 \{1,\ldots,s\} = V_0\sqcup V_1\sqcup \cdots \sqcup V_p
$$
where the $V_y$ are consecutive intervals. 
\bi
We refer the reader to \cite{RS99} for the definition of the combinatorial terms we use in the following.
In each interval, the positions behave like a \textit{Dyck path}. In fact, the positions $\theta_i$ of the upper halves occur as the moves $(1,1)$ and
the positions $\theta_{i+1}$ of the lower halves as the moves $(1,-1)$. We also draw the positions $\theta_j$ of the primary parts as the moves $(1,0)$, and we obtain what is called a \textit{Motzkin path} (also see \cite{DS77}).
With the bijection between \textit{Dyck paths} of length $2l$ and the oriented rooted trees with $l$ egdes,
one can then see the initial positions as an oriented rooted forest with exactly 
$p+1$ trees and $s$ edges. 
\begin{ex}
We take the corresponding representations for  the example  \eqref{example2}. We then have that 
\[(i_1,i_2,i_3,i_4,i_5,i_6)=(1,4,6,8,11,13),\quad (j_0,j_1,j_2,j_3)=(0,3,10,15)\] and 
\[0,\ldots,15 = \theta_{j_0},\theta_{j_1},\theta_{i_1},\theta_{i_1+1},\theta_{i_4},\theta_{i_2},\theta_{i_2+1},\theta_{i_3},\theta_{i_3+1},\theta_{i_4+1},\theta_{j_2},\theta_{i_5},\theta_{i_5+1},\theta_{i_6},\theta_{i_6+1},\theta_{j_3}\]
and the representations correspond to the following diagrams:
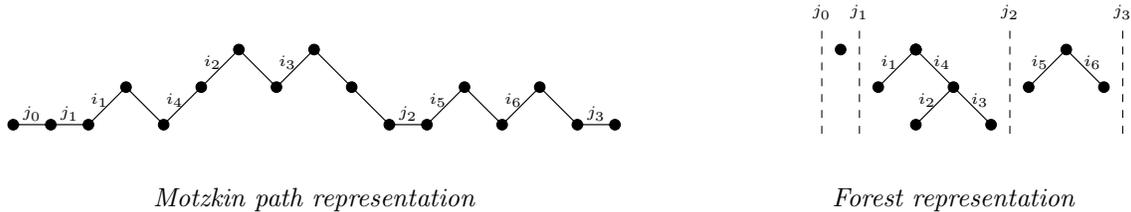
\begin{center}
\begin{tikzpicture}
\filldraw (-0.5,-6) circle (2pt) -- ++(0.5,0)circle (2pt) -- ++(0.5,0)circle (2pt) --++(0.5,0.5) circle (2pt)--++(0.5,-0.5)circle (2pt)--++(0.5,0.5) circle (2pt)--++(0.5,0.5) circle (2pt)--++(0.5,-0.5) circle (2pt)
--++(0.5,0.5) circle (2pt)--++(0.5,-0.5) circle (2pt)--++(0.5,-0.5) circle (2pt)--++(0.5,0) circle (2pt)--++(0.5,0.5) circle (2pt)--++(0.5,-0.5) circle (2pt)--++(0.5,0.5) circle (2pt)--++(0.5,-0.5) circle (2pt)-- ++(0.5,0) circle (2pt);
\draw (3.5,-7) node[align= center] {\textit{Motzkin path representation}};

\draw (-0.25,-5.85) node {{\tiny $j_0$}};
\draw (0.25,-5.85) node {{\tiny $j_1$}};
\draw (4.75,-5.85) node {{\tiny $j_2$}};
\draw (7.25,-5.85) node {{\tiny $j_3$}};

\draw (0.65,-5.67) node {{\tiny $i_1$}};
\draw (1.65,-5.67) node {{\tiny $i_4$}};
\draw (2.15,-5.17) node {{\tiny $i_2$}};
\draw (3.15,-5.17) node {{\tiny $i_3$}};
\draw (5.15,-5.67) node {{\tiny $i_5$}};
\draw (6.15,-5.67) node {{\tiny $i_6$}};

\filldraw (10.5,-5) circle (2pt);
\filldraw (11.5,-5) circle (2pt);
\filldraw (11.5,-5) circle (2pt) -- ++ (-0.5,-0.5) circle (2pt);
\filldraw (11.5,-5) circle (2pt) -- ++ (0.5,-0.5) circle (2pt) -- ++(0.5,-0.5) circle (2pt);
\filldraw (12,-5.5) -- ++ (-0.5,-0.5) circle (2pt) ;
\filldraw (13,-5.5) circle (2pt) -- ++ (0.5,0.5) circle (2pt) -- ++(0.5,-0.5) circle (2pt);
\draw (12,-7) node[align= center] {\textit{Forest representation}};

\draw (10.25,-4.5) node {{\tiny $j_0$}};
\draw (10.75,-4.5) node {{\tiny $j_1$}};
\draw (12.75,-4.5) node {{\tiny $j_2$}};
\draw (14.25,-4.5) node {{\tiny $j_3$}};

\draw [dashed] (10.25,-4.75) -- ++ (0,-1.4);
\draw [dashed] (10.75,-4.75) -- ++ (0,-1.4);
\draw [dashed] (12.75,-4.75) -- ++ (0,-1.4);
\draw [dashed] (14.25,-4.75) -- ++ (0,-1.4);

\draw (11.15,-5.17) node {{\tiny $i_1$}};
\draw (11.65,-5.67) node {{\tiny $i_2$}};
\draw (13.15,-5.17) node {{\tiny $i_5$}};

\draw (11.85,-5.17) node {{\tiny $i_4$}};
\draw (12.35,-5.67) node {{\tiny $i_3$}};
\draw (13.85,-5.17) node {{\tiny $i_6$}};

\end{tikzpicture}
\end{center}
\end{ex}
\bi
Note that while we still keep track of the primary parts as the horizontal moves in Motzkin paths, they \textit{vanish} in oriented rooted forests. However, we can manage to record all information of the partition $\nu$ in the oriented rooted forest by weighting the edges with the corresponding secondary part, while recording each primary part on the root to its right. The optimal forbidden pattern ending by a primary part will then be represented by a weighted oriented rooted tree. 
\m
Let us now consider the edges of the roots. In terms of \textit{Motzkin paths}, they exactly correspond to the meeting points with the horizontal axis. For the final positions, they correspond to the elements
$i\in I$ that satisfy $\theta_{i+1}<\theta_{i'}$ for all $i'>i$.
By \Prp{pr4}, in the case where the Bridge is not a element of $J$, it then corresponds to some root's edge. 
This means that the study of optimal forbidden patterns not ending by a primary part can be reduced to the study of \textit{planted trees} weighted by the secondary parts. The planted trees are indeed in bijection with the oriented trees with one fewer edge, and the problem then becomes the same as the previous case.
\bi To conclude, we see that we can reduce the study of the optimal forbidden patterns to the study of weighted oriented rooted trees, and this give a new perspective to investigate on a precise enumeration of these patterns.
\bi
\section{Proofs of the technical lemmas}
\subsection{Proof of \Lem{lem1}}
To prove \eqref{oe}, we observe that, for any $(l_p,k_{q})\in\Pp\times \Sc$,  by \eqref{Ordd}, 
\[l_p\not \gg k_{q} \Longleftrightarrow  l_p\not\succeq (k+1)_{q}\,,\]
and
\begin{align*}
(k+1)_{q}\gg (l-1)_p&\Longleftrightarrow (k+1)_{q}\succ l_p\\
 &\Longleftrightarrow (k+1)_{q}\not\preceq l_p\,\cdot
\end{align*}
To prove \eqref{eo}, we first remark that, by \eqref{abba},  $\alpha(k_{q})=\beta((k+1)_{q})$. We then obtain by \eqref{Ordd} that
\[l_p \gg \alpha(k_{q}) \Longleftrightarrow (l-1)_p \succeq \alpha(k_{q})\]
and  
\begin{align*}
\beta((k+1)_{q})\not\succ(l-1)_p&\Longleftrightarrow \alpha(k_{q})\not\succ(l-1)_p\\
&\Longleftrightarrow \alpha(k_{q})\preceq (l-1)_p\,\cdot
\end{align*}
\subsection{Proof of \Lem{lem2}}
Let us consider $\min\{k-l: \beta(k_{p})\succ \alpha(l_{q})\}$. 
An abstract way to show \eqref{gam} is to use the explicit formula 
$$
\Delta(p,q) = \chi(r\leq y)+\chi(r\leq x)\chi(s\leq y)
$$
with $q=a_xa_y$ and $p=a_ra_s$. Recall that $x<y$ and $r<s$. In fact, by considering \eqref{orD} and the lexicographic order $\succ$, one can check that the minimal difference between the secondary colors $p$ and $q$ for the relation $\triangleright$ is
\[1+\chi(p\leq q) = 1+\chi(r<x)+\chi(r=x)\chi(s\leq y)\,\cdot\]
By definition \eqref{special},  
\[\chi((p,q)\in \Ccc) = \chi(r>y)+\chi(r<x)\chi(s>y)\]
so that, by \eqref{Ordd}, the minimal difference between the secondary colors $p$ and $q$ for the relation $\gg$ is given by 
\[1+\chi(r<x)+\chi(r=x)\chi(s\leq y)-\chi((p,q)\in \Ccc) = \chi(r\leq y)+\chi(r\leq x)\chi(s\leq y)\,\cdot\]
Now, we reason first according to the parity of $k$. For $k=2u$, we have by \eqref{ab} that $\alpha(k_p)=u_{a_s}$ and $\beta_{k_p}=u_{a_r}$.
In order to minimize $k-l$, $\alpha(l_q)$ and $\beta(l_q)$ have to be the greatest primary parts with color $a_x$ and $a_y$ smaller than $u_{a_r}$ in terms of $\succ$, 
so that, by  \eqref{lex}, they must necessarily be the parts $(u-\chi(r\leq x))_{a_x}$ and $(u-\chi(r\leq y))_{a_y}$. We then obtain the difference
\[\chi(r\leq x)+\chi(r\leq y)\,\cdot\]
With the same reasoning for $k=2u+1$,  since $\alpha(k_p)=(u+1)_{a_r}$ and $\beta(k_p)=u_{a_s}$, we then reach the difference
\[1+\chi(s\leq x)+\chi(s\leq y) \geq \chi(r\leq y)+\chi(s\leq y)\cdot\]
Since the mimimum is reached either for $k$ even or $k$ odd, we then have that 
\[\min\{k-l: \beta(k_{p})\succ \alpha(l_{q})\}\geq \min\{\chi(r\leq y)+\chi(s\leq y),\chi(r\leq x)+\chi(r\leq y)\}\,\cdot\]
We finally consider the case $l=2v$, so that $\alpha(l_q)=v_{a_y}$ and $\beta(l_q)=v_{a_x}$, and to minimize $k-l$, $\alpha(k_p)$ and $\beta(k_p)$ have to be the smallest primary parts with color $a_r$ and $a_s$ greater than $v_{a_y}$ in terms of $\succ$, 
so that they must necessarily be the parts $(v+\chi(r\leq y))_{a_r}$ and $(v+\chi(s\leq y))_{a_s}$. We obtain the difference $\chi(r\leq y)+\chi(s\leq y)\,$
and then the inequality 
\[\min\{k-l: \beta(k_{p})\succ \alpha(l_{q})\}\leq \min\{\chi(r\leq y)+\chi(s\leq y),\chi(r\leq x)+\chi(r\leq y)\}\cdot\]
Since $\min\{\chi(r\leq y)+\chi(s\leq y),\chi(r\leq x)+\chi(r\leq y)\}=\chi(r\leq y)+\chi(r\leq x)\chi(s\leq y)$, we then have \eqref{gam}.
\\\\To prove \eqref{sw1}, 
we have by \eqref{abba} that $\alpha((l-1)_q)=\beta(l_q)$. Since $\beta(k_p)\succ \beta(l_q)=\alpha((l-1)_q)$, this then implies by \eqref{gam} that $k_p\gg (l-1)_q$, and this is equivalent to $(k+1)_p\gg l_q$.
\\\\Let us now suppose that $k-l\geq \Delta(p,q)$. We just saw that this minimum value was reached at $k$ or $k-1$. Then if we do not have $\beta(k_{p})\succ \alpha(l_{q})$, we necessarily have $\beta((k-1)_{p})\succ \alpha((l-1)_{q})=\beta(l_{q})$ by \eqref{abba}. Moreover, by \eqref{Ordd}, we have
\[\beta(k_{p})\not\succ \alpha(l_{q})\Longleftrightarrow \alpha(l_{q})+1\gg \alpha((k-1)_{p} )\,,\]
so that we obtain \eqref{sw}. Suppose now that we have $k-l= \Delta(p,q)$. If $\beta(k_{p})\succ \alpha(l_{q})$ then we necessarily have 
\[\beta(k_{p})\succ \alpha(l_{q})\succ \beta(l_q)\succeq \beta(k_p)-1\,\cdot\]
In fact, we saw that the minimal difference is obtained when the primary parts $\alpha(l_{q})$ and $\beta(l_q)$ are the closest possible to $\beta(k_{p})$ with the primary colors of $q$. 
If $\beta(k_{p})\not\succ \alpha(l_{q})$, since we have $\beta(l_q)+1\succ\alpha(l_{q})$, we also have 
\[\beta(l_q)+1\succ\alpha(l_{q})\succeq \beta(k_{p})\,\cdot\]
In both cases, the relation \eqref{chaine} holds. If we have that $k-l-1\geq \Delta(p,q)$, then we necessarily have by \eqref{abba} that 
\[\beta(k_q)\succ \beta{(l+1)_{q}} = \alpha(l_q)\,\cdot\]
\subsection{Proof of \Lem{lem3}}
For any $\nu =(\nu_1,\ldots,\nu_t)\in \Eee$ and any $i\in [1,t-2]$, we have 
$$\nu_i\triangleright \cdots \triangleright\nu_{j}\,\cdot$$
By \eqref{Ord}, we have 
\[\nu_i\succeq \nu_{i+1}+1 \succeq \cdots\succeq \nu_{j}+j-i \Rightarrow \nu_i \succeq \nu_{j}+j-i\,,\]
with a strict inequality as soon as we have $\nu_i$ or $\nu_j$ in $\Sc$, and we thus obtain \eqref{ee}.
\subsection{Proof of \Lem{lem4}}
By definition, for all $i\in I$, $\Br(i)\in([i,j)\cap I)\cup \{j\}$, for $j = \min(i,p+2s+1]\cap J$.
This means that, for any $I \ni i'>j$, $$\Br(i')\geq i'>j\geq \Br(i)\,\cdot$$
Let us now consider the function $\Br$ on $[i,j)\cap I$. 
It is obvious that, for all $i'\in [i,j)\cap I$, we have 
$j=\min(i',p+2s+1]\cap J$. 
\begin{itemize}
 \item If $\Br(i)=i$, then $$\Br(i')\geq i'\geq i = \Br(i)\,\cdot$$
 \item If we have $\Br(i)=j$, then by \eqref{br1}
$$\nu_{u+1}\not \succ \nu_j+\frac{j-u}{2}-1$$
for all $u\in[i,j)\cap I$, and since $[i',j)\subset [i,j)$, we also obtain that $\Br(i')=j$.
\item Finally, if  $\Br(i)\in (i,j)\cap I$, then  we have either $j> i'\geq \Br(i)$, or
$i\leq i'<\Br(i)$. In the first case, we obtain 
$$\Br(i')\geq i' \geq \Br(i)\,\cdot$$
In the second case, we observe that, by \eqref{br2} and \eqref{br3},
$$\nu_{u+1}\not \succ \nu_{\Br(i)}+\frac{\Br(i)-u}{2}-1$$
for all $u\in[i,\Br(i))\cap I$, and in particular for all  $u\in[i',\Br(i))\cap I$. Thus, if $\Br(i')\neq j$, we necessarily have by \eqref{br3} that $\Br(i')\geq \Br(i)$.
\end{itemize}
In any case, we have that $\Br(i')\geq \Br(i)$. 
\\\\Let us now suppose that $\Br(i)\in I$. If $\Br(i)=i$, then $\Br(\Br(i))=i=\Br(i)$. Otherwise, let us assume that 
$\Br(\Br(i))>\Br(i)$.
\begin{itemize}
 \item If $\Br(\Br(i))=j$, this means that 
 $$\nu_{u+1}\not \succ \nu_j+\frac{j-u}{2}-1 \Longleftrightarrow  \nu_j+\frac{j-u}{2}-1 \succeq \nu_{u+1}$$
 for all $u\in[\Br(i),j)\cap I$. Since $\nu_{\Br(i)}$ and $\nu_{\Br(i)+1}$ have different primary colors and  are consecutive with respect to $\succ$, we then obtain that
 $\nu_{\Br(i)+1}+1\succ\nu_{\Br(i)}$, so that
 $$\nu_j+\frac{j-\Br(i)}{2} \succeq \nu_{\Br(i)}\,\cdot$$
 We also have by \eqref{br2} and \eqref{br3} that 
 $$\nu_{u+1}\not \succ \nu_{\Br(i)}+\frac{\Br(i)-u}{2}-1 \Longleftrightarrow  \nu_{\Br(i)}+\frac{\Br(i)-u}{2}-1 \succeq \nu_{u+1}$$
for all $u\in[i,\Br(i))\cap I$, so that
$$\nu_j+\frac{j-u}{2}-1 \succeq \nu_{u+1} \Longleftrightarrow  \nu_{u+1}\not \succ \nu_j+\frac{j-u}{2}-1\,\cdot$$
We then conclude by \eqref{br1} that $\Br(i)=j$, which contradicts the fact that $\Br(i)\notin J$.
\item For $\Br(\Br(i))>\Br(i)$, we reason exactly in the same way, by replacing $j$ by $\Br(\Br(i))$, and we obtain by \eqref{br3} that 
$\Br(i)\geq \Br(\Br(i)) > \Br(i)$.
\end{itemize}
To conclude, we necessarily have that $\Br(\Br(i))=\Br(i)$ for $\Br(i)\in I$.
\subsection{Proof of \Lem{lem5}}
 By \eqref{Ord}, \eqref{Ordd} and the fact that all the pairs in $\Ccc$ have distinct secondary colors, we have that for any $u\in [i,i')\cap I$
 \[
 \nu_{u+2}+\nu_{u+3}+1 \succ \nu_{u}+\nu_{u+1}\succ \nu_{u+2}+\nu_{u+3}\,,
\]
so that we obtain \eqref{X} recursively.
\subsection{Proof of \Lem{lem6}}
By \eqref{chaine} of \Lem{lem2}, we have for any $u\in [i,i')\cap I$ that
\[\nu_{u+3}+1\succeq \nu_{u+1}\,,\]
so that we recursively have 
\[\nu_{i'+1}+\frac{i'-u}{2}\succeq \nu_{u+1}\,\cdot \]
By \eqref{cross}, if we suppose that $\Br(i')>i'$, we then have 
\[\nu_{i'+1}\not \succ \nu_{\Br(i')} +\frac{\Br(i')-i'}{2}-1 \Longleftrightarrow \nu_{\Br(i')} +\frac{\Br(i')-i'}{2}-1 \succeq  \nu_{i'+1}\,,\]
and we obtain that 
$$
 \nu_{\Br(i')} +\frac{\Br(i')-u}{2}-1 \succeq  \nu_{u+1} \Longleftrightarrow \nu_{u+1}\not \succ \nu_{\Br(i')} +\frac{\Br(i')-u}{2}-1
$$
for all $u\in [i,i')\cap I$. Since the previous relation is also true for for all $u\in [i',\Br(i'))\cap I$, by \eqref{br2} and \eqref{br3}, we have that $\Br(i')\leq \Br(i)$. 
Finally, by \Lem{lem4}, the fact that $\Br$ is non-decreasing on $I$ gives that $\Br(i')= \Br(i)$.
\subsection{Proof of \Lem{lem7}}\label{l7}
We can notice that for any pair $(k_p,l_q)$ of secondary parts different from a pattern $cd \rightarrow ab$ and that satisfies $k_p\gg l_q$, we can always find  an integer $h$ such 
\begin{equation}\label{inser}
 k_p\succ h_{cd} \succeq l_p\,\cdot
\end{equation}
This is obvious when $(p,q)\notin \Ccc$. In fact,
\[k_p\gg l_q \Longleftrightarrow k_p\succ (l+1)_q \succ l_q\]
and we can find a unique $h_{cd}$ satisfying $(l+1)_q \succ h_{cd} \succeq l_q$. Note that if $q=cd$, we then have at least two possible integers $h = l,l+1$. 
Suppose now that $(p,q)\in \Ccc$. Recall that here, we set $\{a_1<a_2<a_3<a_4<a_5\} = \{a<b<c<d<e\}$. We then have two kinds of pairs.
\begin{itemize}
 \item First, we have the pairs $(a_ia_j,a_ka_l)$ with $5\geq j>i>l>k \geq 1$, so that $i\geq 3$ and $l\leq 3$. Thus, $a_ia_j\geq cd$, while $a_ka_l\leq bc<cd$.
 If $a_ia_j\neq cd$, we have that $a_ia_j>cd$, and then 
 \[k_{a_ia_j}\succ k_{cd}\succ k_{a_ka_l}\]
 and the property \eqref{inser} is true for $(k_p,l_q) = (k_{a_ia_j}, k_{a_ka_l})$.
 \item The second kind of pair is of the form $(a_ia_j,a_ka_l)$ with $ 5\geq j>l>k>i\geq 1$, so that $l\leq 4$ and $i\leq 2$. Thus, $a_ia_j\leq be<cd$, while $a_ka_l\leq cd$.
 We have that $a_ia_j>cd$, and then 
 \[(l+1)_{a_ia_j}\succ l_{cd}\succeq l_{a_ka_l}\]
 and the property \eqref{inser} is true for $(k_p,l_q) = ((l+1)_{a_ia_j}, l_{a_ka_l})$.
\end{itemize}
Let us now consider a pattern of secondary parts $(\nu_1,\nu_2,\cdots,\nu_{2s-1},\nu_{2s},\nu_{2s+1},\nu_{2s+2})$ with no moves $\rightarrow cd \rightarrow$. 
\m
If $\nu_1+\nu_2,\nu_3+\nu_4 \neq cd \rightarrow ab$,
we recursively show on $1\leq i\leq s$ that there exists $h$ such that 
\begin{equation}\label{inser2}
 \nu_1+\nu_2\succ (h+i-1)_{cd}\succ h_{cd}\succeq \nu_{2i+1}+\nu_{2i+2}\,\cdot
\end{equation} 
In fact, by \eqref{inser}, the previous statement holds for $i=1$. Suppose now it holds by induction up to $i$. 
If $\nu_{2i+1}+\nu_{2i+2},\nu_{2i+3}+\nu_{2i+4} \neq cd \rightarrow ab$, then by \eqref{inser}, we have $h'$ such that
\[h_{cd}\succeq \nu_{2i+1}+\nu_{2i+2}\succ h'_{cd}\succeq \nu_{2i+3}+\nu_{2i+4}\,\cdot\]
We thus have $h>h'$, and by choosing $h'$, we obtain 
\[\nu_1+\nu_2\succ (h'+i)_{cd}\succ h'_{cd}\succeq \nu_{2i+3}+\nu_{2i+4}\,\cdot\]
If $\nu_{2i+1}+\nu_{2i+2},\nu_{2i+3}+\nu_{2i+4} = cd \rightarrow ab$, we then necessarily have that $\nu_{2i-1}+\nu_{2i}\triangleright\nu_{2i+1}+\nu_{2i+2}$ not to have the moves $\rightarrow cd \rightarrow$.
Therefore, by setting $h_{cd}=\nu_{2i+1}+\nu_{2i+2}$, we have that $\nu_{2i-1}+\nu_{2i}\succ (h+1)_{cd}$. 
Since the statement \eqref{inser2} also holds for $i-1$, there exists $h'$  such that
\[\nu_1+\nu_2\succ (h'+i-2)_{cd}\succ h'_{cd}\succeq \nu_{2i-1}+\nu_{2i}\,\cdot\]
We can then remark that $h'\geq h+2$, and we conclude that
\[\nu_1+\nu_2\succ (h+i)_{cd}\succ h_{cd}\succeq \nu_{2i+3}+\nu_{2i+4}\,\cdot\]
We have thus proved the statement \eqref{inser2} when the head is different from $cd \rightarrow ab$. 
\m
If the head is equal to $cd \rightarrow ab$, we then apply \eqref{inser2} on 
the pattern $(\nu_3,\nu_4,\cdots,\nu_{2s-1},\nu_{2s},\nu_{2s+1},\nu_{2s+2})$, and we obtain that there exists $h$ such that
$$
 \nu_1+\nu_2 \succ \nu_2+\nu_3\succ (h+i-2)_{cd}\succ h_{cd}\succeq \nu_{2i+1}+\nu_{2i+2}
$$
so that $\nu_1+\nu_2 \succeq (h+i-1)_{cd}$.
In both cases, we always have that $\nu_1+\nu_2 -s+1\succeq \nu_{2s+1}+\nu_{2s+2}$ so that
$$
 \overline{\nu_1+\nu_2 -s+1}\succeq \overline{\nu_{2s+1}+\nu_{2s+2}}\,\cdot
$$
By definition \eqref{kill}, $(\nu_1,\nu_2,\cdots,\nu_{2s-1},\nu_{2s},\nu_{2s+1},\nu_{2s+2})$ cannot be a shortcut. Since a pattern that does not contain the moves $\rightarrow cd \rightarrow$
does not have any subpattern that contains these moves, we then obtain our lemma.
\section{Proofs of the propositions}
\subsection{Proof of \Prp{pr1}}\label{pp1}
Let $\la = (\la_1,\ldots,\la_t)$ be a partition in $\Od$. Let us set $c_1,\ldots,c_t$ to be the primary colors of the parts $\la_1,\ldots,\la_t$.
\paragraph{\underline{\textbf{First} \Soo}} Now consider the first troublesome pair $(\la_i,\la_{i+1})$ at \So in $\Phi$. We then set 
\begin{align*}
\delta^1&=\emptyset\\
\gamma^1&=\la_1\gg\cdots\gg \la_i\,,\\
\mu^1&=\la_{i+1}\succ \cdots\succ\la_t\,\cdot
\end{align*}
The first resulting secondary part is $\la_i+\la_{i+1}$.
\paragraph{\underline{\textbf{First iterations of} \Stt}}
\begin{itemize}
\item If there is a part $\la_{i+2}$ after $\la_{i+1}$, we have that
\begin{align*}
\la_i+\la_{i+1}-\la_{i+2}&= \chi(c_i<c_{i+1})+2\la_{i+1}-\la_{i+2}& \text{by \eqref{nog}}\\
&\geq \chi(c_i<c_{i+1})+2\chi(c_{i+1}\leq c_{i+2})+\la_{i+2}&\text{by \eqref{lex}}\\
&\geq 1+\chi(c_i\leq c_{i+2})+\chi(c_{i+1}\leq c_{i+2})\,\cdot
\end{align*}
Since by \eqref{orD}, we have that $c_i> c_{i+2}$ and $c_{i+1}> c_{i+2}$ implies $c_ic_{i+1}>c_{i+2}$, we then have that
$\la_i+\la_{i+1}-\la_{i+2}\geq 1+\chi(c_ic_{i+1}\leq c_{i+2})$, and we conclude that $\la_i+\la_{i+1}\gg \la_{i+2}$. This means that if there is no iteration of \St (which happens if $i=1$ or $\la_{i+1}\gg \la_i+\la_{i+1}$), then the secondary part
is well-ordered with the primary part to its right.\\
\item The primary parts of $\gamma^{1}$ are well-ordered by $\gg$.
By \eqref{Ordd} and  \eqref{aj}, we have that for any $j<i$, 
if $\la_i+\la_{i+1}$ crosses
$\la_j$ after $i-j$ iterations of \Stt, we then have  by \eqref{oe} that
$$
(\la_i+\la_{i+1}+i-j)\gg (\la_j-1)\gg\cdots\gg\la_{i-1}-1\,\cdot
$$
\item We also have by \eqref{Ordd} that 
\begin{align*}
\la_{i-1}\gg\la_i\succ\la_{i+1}\succ \la_{i+2} &\Longrightarrow \la_{i-1}-1\succeq \la_i\succ\la_{i+1}\succ \la_{i+2}\\
&\Longrightarrow \la_{i-1}-1\succ\la_{i+2}\,\cdot
\end{align*}
\end{itemize}
If we can no longer apply \St after $i-j$ iterations,  we then obtain (even when there is no crossing which means that $j=i$)
$$
\la_1\gg\cdots\gg\la_{j-1}\gg(\la_i+\la_{i+1}+i-j)\gg (\la_j-1)\gg\cdots\gg\la_{i-1}-1\succ \la_{i+2}\succ\cdots \succ \la_t\cdot
$$
\paragraph{\underline{\textbf{Second} \Soo}} Now, by applying \So for the second time, we see that the next troublesome pair is either $\la_{i-1}-1,\la_{i+2}$, or  $\la_{i+2+x},\la_{i+3+x}$ for some $x\geq 0$.
\begin{itemize}
\item If $\la_{i-1}-1\not \gg\la_{i+2}$, this means that $(\la_{i-1}-1,\la_{i+2})$ is a troublesome pair, and  \So occurs there. We then set 
\begin{align*}
\delta^2&=\la_1\gg\cdots\gg\la_{j-1}\gg(\la_i+\la_{i+1}+i-j)\\
\gamma^2&=(\la_j-1)\gg\cdots\gg\la_{i-1}-1\\
\mu^2&=\la_{i+2}\succ\cdots \succ \la_t\,\cdot
\end{align*}
By \eqref{sw1}, we have that 
$(\la_i+\la_{i+1}+1)\gg (\la_{i-1}+\la_{i+2}-1)$.
Then, even if $(\la_{i-1}+\la_{i+2}-1)$ crosses the primary parts $(\la_j-1)\gg\cdots\gg\la_{i-2}-1$ after $i-j-1$ iterations of \Stt, by \eqref{Ordd}, we will still have that 
\[(\la_i+\la_{i+1}+i-j)\gg(\la_{i-1}+\la_{i+2}+i-j-2)\,\cdot\]
We have before the third application of \So that 
\begin{align*}
\delta^3&=\la_1\gg\cdots\gg(\la_i+\la_{i+1}+i-j)\gg \la_j-1\gg\cdots\gg\la_{j'-1}-1\gg(\la_{i-1}+\la_{i+2}-2+i-j')\\
\gamma^3,\mu^3&=\la_{j'}-2\gg\cdots\gg\la_{i-2}-2\gg\la_{i+3}\succ\cdots\succ \la_t \,,
\end{align*}
for some $i-1 \geq j' \geq j$. Observe that $\mu^3$ is the tail of the partition $\la_{i+3}\succ\cdots\succ \la_t$.\\
\item If $\la_{i-1}-1 \gg\la_{i+2}$, then the next troublesome pair appears at 
$\la_{i+2+x},\la_{i+3+x}$ for some $x\geq 0$, and it forms the secondary part $\la_{i+2+x}+\la_{i+3+x}$.
 We then set 
\begin{align*}
\delta^2&=\la_1\gg\cdots\gg\la_{j-1}\gg(\la_i+\la_{i+1}+i-j)\\
\gamma^2&=(\la_j-1)\gg\cdots\gg\la_{i-1}-1\gg \la_{i+2}\gg \cdots \gg\la_{i+2+x}\\
\mu^2&=\la_{i+x+3}\succ\cdots \succ \la_t\,\cdot
\end{align*}
We also have 
\begin{equation*}
\la_i\succ\la_{i+1}\succ\la_{i+2}\gg \cdots \gg \la_{i+2+x}\succ\la_{i+3+x}\,\cdot
\end{equation*}
By \eqref{Ordd}, we can easily check that
\[
\la_i\succ\la_{i+1}\succ\la_{i+2} \succeq \la_{i+2+x}+x\succ\la_{i+3+x}+x
\]
so that, by \eqref{gam}, 
\[
(\la_i+\la_{i+1})\gg (\la_{i+2+x}+\la_{i+3+x}+2x)\,\cdot
\]
This means by \eqref{Ordd} that, 
\begin{equation*}
(\la_i+\la_{i+1})\gg (\la_{i+2+x}+\la_{i+3+x}+x)
\end{equation*}
and, as soon as $x\geq 1$, by \eqref{Ord}
\[
(\la_i+\la_{i+1})\triangleright (\la_{i+2+x}+\la_{i+3+x}+x)\, \cdot
\]
We then obtain that, even if the secondary part $\la_{i+2+x}+\la_{i+3+x}$ crosses,  after $x+i-j$ iterations of \Stt, the primary parts
\[\la_j-1\gg\cdots\gg (\la_{i-1}-1)\gg \la_{i+2}\gg \cdots \gg \la_{i+1+x}\,,\]
we still have
\[(\la_i+\la_{i+1}+i-j)\gg(\la_{i+2+x}+\la_{i+3+x}+x+i-j)\,\cdot\]
However, as soon as $x\geq 1$, we directly have 
$$
(\la_i+\la_{i+1}+i-j)\triangleright(\la_{i+2+x}+\la_{i+3+x}+x+i-j)\,\cdot
$$
We thus obtain before the third application of \So that, 
\begin{align*}
\delta^3&=\la_1\gg\cdots\gg(\la_i+\la_{i+1}+i-j)\gg\cdots\gg (\la_{i+2+x}+\la_{i+3+x}+x+i-j')\\
\gamma^3,\mu^3&=\cdots\succ\la_{i+4+x}\succ\cdots\succ\la_t\,,
\end{align*}
for some $i+x \geq j'\geq  j$. Observe that $\mu^3$ is the tail of the partition $\la_{i+x+3}\succ\cdots\succ \la_t$. Moreover, we have the following inequalities 
\begin{itemize}
 \item $\la_{j'-1}-1\gg(\la_{i+2+x}+\la_{i+3+x}+x+i-j')\gg\la_{j'}-2$ for $x-1\geq j'\geq j$,
 \item $\la_{i-1}-1\gg(\la_{i+2+x}+\la_{i+3+x}+x)\gg\la_{i+2}-1$ for $j'=i$,
 \item $\la_{j'+1}\gg(\la_{i+2+x}+\la_{i+3+x}+x+i-j')\gg\la_{j'+2}-1$ for $x+i\geq j'\geq i+1$.
\end{itemize}
Observe that the partition to the left of $\la_{i+x+4}$ is well-ordered by $\gg$, so that $\mu^3$ is the tail of the partition $\la_{i+x+4}\succ\cdots\succ \la_t$.
\end{itemize}
In both cases, the conditions in the proposition are satisfied. In fact, the partition $\delta^2$ belongs to $\E$ and is the head of the partition $\delta^{3}$ that also belongs to $\E$, and the fourth statement is true.
By comparing $\mu^1,\mu^2$ (and $\mu^3$), the third statement is true since $\mu^2$ is a strict tail of $\mu^1$.
The two first statements directly come from the way we established the sequences, and the fact that $s(\delta^u)\gg g(\gamma^u)$ is true for $u=2,3$.
\bi
By induction,  we only apply  \So once to the troublesome pair $(s(\gamma^{u}),g(\mu^u))$ in the partition $\emptyset,\gamma^u,\mu^u\in \Od$  and  then some iterations of \Stt. We then obtain some sequence 
$\delta'^u,\gamma'^u,\mu'^u$ with the same form as $(\delta^2,\gamma^2,\mu^2)$, and we set the triplet $(\delta^{u+1},\gamma^{u+1},\mu^{u+1}) = ((\delta^u,\delta'^u),\gamma'^u,\mu'^u)$.
Note that the sequence $\delta^u,\delta'^u$ is indeed a partition in $\E$ by considering the process from the $(u-1)^{th}$ \Soo. 
\\Then, the sequence $(\delta^u,\gamma^u,\mu^u)$ becomes the sequence $(\delta^{u+1},\gamma^{u+1},\mu^{u+1})$ after applying 
\So once to the troublesome pair $(s(\gamma^u),g(\mu^u))$, and some iterations of \St by crossing the secondary part $s(\gamma^u)+g(\mu^u)$ with some primary parts of $\gamma^u\setminus \{s(\gamma^u)\}$.
\Prp{pr1} follows naturally.
\subsection{Proof of \Prp{pr2}}\label{pp2}
Let us consider $\E\ni \nu=(\nu_1,\ldots,\nu_t)$. If we suppose that the secondary parts of $\nu$ are $\nu_{i_1},\ldots,\nu_{i_S}$ for $i_1<\cdots<i_S$,   
we can then set for all $v\in [1,S]$
$$
 \delta^v = \nu_1\gg \cdots\gg \nu_{i_{S+1-v}}
$$
and $\delta^{S+1}=\emptyset$.
By setting $i=i_S$, we also have that
\begin{align*}
\delta^1&=\nu_1\gg\cdots\gg\nu_i\\
\gamma^1 &= \nu_{i+1}\gg \cdots\gg\nu_t\\
\mu^1&=\emptyset\,\cdot
\end{align*}
\begin{itemize}
\item If $\nu_i$ crosses all the primary parts up to $\nu_t$ after iterating \Soo, we have that 
\[\beta(\nu_i-t+i+1)\not\succ \nu_t \,\cdot\]
But, we also have that
\[\nu_i\triangleright \cdots \triangleright \nu_t\]
since $\nu_{i+1},\ldots,\nu_t$ are all primary parts. We thus have by \Lem{lem3} that
\[\nu_i-t+i\succ \nu_t\,,\]
so that, if $\nu_i-t+i$ has size $1$, then $\nu_t$ has also size $1$ and a color smaller than the color of $\nu_i$. But by \eqref{half1} and \eqref{orD},
we necessarily have that $\beta(\nu_i-t+i+1)$ has size $1$ and a color greater than 
the color of $\nu_i$. We then obtain by \eqref{cons1} that 
\[\beta(\nu_i-t+i+1)\succ \nu_i-t+i\succ \nu_t\,,\]
and we do not cross $\nu_i-t+i+1$ and $\nu_t$, which is aburd by assumption.
This means that in any case after crossing, we still have that the secondary part size is greater than $1$, so that after splitting, its upper and lower halves stay in $\Pp$.\\ 
\item if $\nu_i$ crosses all the primary parts up to $\nu_j$ after iterating \So and stops before $\nu_{j+1}$, we then set 
\begin{align*}
\delta^2&=\nu_1\gg\cdots\gg\nu_{i_{S-1}}\\
\gamma^2 &= \nu_{i_{S-1}+1},\ldots,\nu_{i_{S}-1},\nu_{i_{S}+1}+1,\ldots, \nu_{j}+1, \alpha(\nu_{i_S}+i_S-j)\\
\mu^2&=\beta(\nu_{i_S}+i_S-j),\nu_{j+1},\ldots,\nu_t\,\cdot
\end{align*}
The statements of \Prp{pr2} are then satisfied.
\item Suppose now that $(\delta^{v},\gamma^{v},\mu^{v})$ satisfies the conditions in \Prp{pr2}. Note that $s(\gamma^{v}),g(\mu^{v})$ are respectively the upper and the lower halves after the 
splitting of the secondary part coming from $\nu_{i_{S+2-v}}$.
We also have by \eqref{Ordd} that 
\[\nu_{i_{S+1-v}}\gg \nu_{i_{S+1-v}}\gg\cdots\gg \nu_{i_{S+2-v}}\gg\nu_{i_{S+2-v}} \Longrightarrow \, \nu_{i_{S+1-v}}+i_{S+1-v}-i_{S+2-v}+1\gg\nu_{i_{S+2-v}}\]
since the parts between these secondary parts are  primary parts. By \Lem{lem2}, even if these secondary parts meet after crossing 
the  primary parts, the splitting of the part coming from $\nu_{i_{S+1-v}}$ will then occur either before the upper half or between the upper and 
the lower halves obtained after the splitting of $\nu_{i_{S+2-v}}$. Thus the splitting of $s(\delta^v)=\nu_{i_{S+1-v}}$ occurs before $g(\mu^v)$.
By taking $s(\gamma^{v+1}),g(\mu^{v+1})$ as the upper and the lower halves of the split secondary part coming from $\nu_{i_{S+2-v}}$, we thus obtain a sequence $(\delta^{v+1},\gamma^{v+1},\mu^{v+1})$ such that $\mu^{v}$ is the strict tail of $\mu^{v+1}$.
Note that these sequences also satisfy the other statements.  
\end{itemize}
\subsection{Proof of \Prp{pr10}} 
Note that \So of $\Phi$ is reversible by the splitting in \St of $\Psi$. Let us now show that iterations of \St of $\Phi$ are also reversible by iterations of \So in $\Psi$.
\\\\We saw in the proof of \Prp{pr1} in \eqref{pp1} that for any $u\geq 1$, the sequence $(\delta^u,\gamma^u,\mu^u)$ becomes the sequence $(\delta^{u+1},\gamma^{u+1},\mu^{u+1})$ after applying
\So once to the troublesome pair $(s(\gamma^u),g(\mu^u))$, and some iterations of \St by crossing the secondary part $s(\gamma^u)+g(\mu^u)$ with some primary parts of $\gamma^u\setminus \{s(\gamma^u)\}$.
Without loss of generality, let us set 
\begin{align*}
\gamma^u &= \pi_1\gg \cdots \gg \pi_i\\
\mu^u&=\pi_{i+1}\succ \cdots\succ \pi_r\,
\end{align*}
and suppose that the secondary parts $\pi_i+\pi_{i+1}$ crossed the primary parts $\pi_j\gg \cdots \gg \pi_{i-1}$. Since $\gamma^u\in \E\cap \Od \subset \Eee$, by \Lem{lem3} and \eqref{aj}, we have that 
\[\pi_j\gg \pi_i+i-j-1 \succeq \alpha(\pi_i+\pi_{i+1}+i-j-1)\,\cdot\]
Using \eqref{eo} of \Lem{lem1}, this is equivalent to saying that
\begin{equation}\label{reverse1}
 \alpha(\pi_i+\pi_{i+1}+i-j)\not \succ \pi_j-1\,\cdot
\end{equation}
If the iteration of \St ceases before $\pi_{k-1}$, we  then have that
\begin{align*}
\delta'^u &= \pi_1\gg \cdots \gg \pi_i+\pi_{i+1}+i-k\\
\gamma'^u,\mu'^u&= \pi_j-1\gg \cdots \gg \pi_{i-1}-1\succ \pi_{i+2}\succ \cdots\succ\pi_r
\end{align*}
so that 
$(\delta^{u+1},\gamma^{u+1},\mu^{u+1})= ((\delta^u,\delta'^u),\gamma'^u,\mu'^u)$. But the inequality  \eqref{reverse1} holds for all $k\leq j\leq i-1$,
 so that by applying $\Psi$ on  $(\delta^{u+1},\gamma^{u+1},\mu^{u+1})$, the secondary part $s(\delta^{u+1})=\pi_i+\pi_{i+1}+i-k$ will recursively cross by \So the parts $\pi_j-1$. 
 The iteration of \So stops before the part $ \pi_{i+2}$ since 
 \[\pi_{i+2}\prec \pi_{j+1} = \beta(\pi_i+\pi_{i+1})\]
 and we split by \St the secondary part  $\pi_i+\pi_{i+1}$ into $\pi_i$ and $\pi_{i+1}$. We then retrieve the sequence $(\delta^u,\gamma^u,\mu^u)$.
 \\\\To conclude, we observe that if $\Phi(\la)\in \E$ has $S$ secondary parts, then the last sequence in the process $\Phi$ is $(\delta^{S+1},\gamma^{S+1},\mu^{S+1})$
 with $\mu^{S+1} = \emptyset$, $\delta^{S+1}$ the partition $\Phi(\la)$ up to the $S^{th}$ secondary part and $\gamma^{S+1}$ the tail to the right of this last secondary part.
 But this triplet is equal to the triplet  $(\delta^v,\gamma^v,\mu^v)$ of \Prp{pr2} for $v=1$. We then recursively obtain the result of \Prp{pr10} in the decreasing order according to $u$.
\subsection{Proof of \Prp{pr5}}
Let us take any $i\in I =\{i_1<\cdots<i_s\}$, let us consider $j = \min(i,p+2s+1]\cap J$. Since in the process of $\Psi$, the primary parts never cross, and the secondary parts can only move forward before splitting,
the part $\nu_j$ will not be affected by $\Psi$ operating on any secondary part to its right. 
\begin{itemize}
 \item Suppose that $\Br(i)=j$. By definition \eqref{br1}, this means that 
 \[\nu_{i'+1}\not \succ \nu_j+\frac{j-i'}{2}-1\]
 for all $i'\in [i,j)\cap I$, so that, by the crossing condition of \St of $\Psi$, $\nu_j+\frac{j-i'}{2}-1$ will recursively be the first primary part that crosses all the secondary parts
$\nu_{i'}+\nu_{i'+1}$ up to $\nu_i+\nu_{i+1}$. Thus, for $i=i_u$
\[s(\delta^{s+1-u})=\nu_{i_u}+\nu_{i_u+1},\quad g(\gamma^{s+1-u})=\nu_j+\frac{j-i_u}{2}-1\,\cdot\]
\item Suppose that $\Br(i)<j$. Let us set $\Br(i)=i'_1$ and  let $i'_1<\cdots<i'_t<j$ be all the fixed points by $\Br$ in $[i,j)$. By \Lem{lem4},  have that 
$$\quad\quad\quad\Br([i,i'_1])=\{i'_1\}\,,\quad \Br((i'_{s-1},i'_s])=\{i'_s\}\quad \text{and for}\quad(i'_t,j)\neq \emptyset\,,\quad \Br((i'_t,j))=\{j\}\,\cdot$$
We then have during the process of $\Psi$ that $\nu_j$ crosses all the secondary parts up to $\nu_{i'_t+2}+\nu_{i'_t+3}$, but does not cross $\nu_{i'_t}+\nu_{i'_t+1}$.
Thus, $\nu_{i'_t}+\nu_{i'_t+1}$ directly splits  into $\nu_{i'_t}$ and $\nu_{i'_t+1}$, and by \eqref{br2} and the crossing condition of \So, $\nu_{i'_t}$ crosses
all the secondary parts up to $\nu_{i'_{t-1}}+\nu_{i'_{t-1}+1}$, which is not crossed.
\\The process then continues and we reach $\nu_{i'_1}+\nu_{i'_1+1}$ which directly splits into $\nu_{i'_1}$ and $\nu_{i'_1+1}$. If $i=i_1$, we have the first statement of \Prp{pr5}. Otherwise,
 $\nu_{i'_1}$ crosses all the secondary parts up to $\nu_{i}+\nu_{i+1}$. We then obtain for for $i=i_u$
\[s(\delta^{s+1-u})=\nu_{i_u}+\nu_{i_u+1},\quad g(\gamma^{s+1-u})=\nu_{i'_1}+\frac{i'_1-i_u}{2}-1\,\cdot\]
\end{itemize}
In any case, if $i=\Br(i)$, then $\nu_i+\nu_{i+1}$ directly splits, otherwise, we have that for $i=i_u$
\[g(\gamma^{s+1-u})=\nu_{\Br(i_u)}+\frac{\Br(i_u)-i_u}{2}-1\,,\]
and the part $\nu_{i_u}+\nu_{i_u+1}$ first crosses the primary part $g(\gamma^{s+1-u})$. 
\subsection{Proof of \Prp{pr6}}\label{pp6}
Let us take $\nu = (\nu_1,\cdots,\nu_{p+2s})$, and $I=\{i_1<\cdots<i_s\}$.
Note that the triplet $(\delta^1,\gamma^1,\mu^1)$ is such that $\mu^1=\emptyset$, $\delta^1$ is the partition $\nu$ up to $\nu_{i_s}+\nu_{i_s+1}$ and $\gamma^1$ is the tail to the right of this part.
We then have that $\gamma^1,\mu^1 \in (\E\cap \Od)\times \Od$ and $\nu_{i_s}+\nu_{i_s+1}=s(\delta^1)\gg g(\gamma^1)$.
\begin{itemize}
 \item If we have that $\Br(i_u)>i_u$, by \Prp{pr5}, we necessarily have that 
\[
g(\gamma^{s+1-u}) = \nu_{\Br(i_u)}+\frac{\Br(i_u)-i_u}{2}-1\,\cdot
\]
But with the condition $(2)$, we have by \eqref{Ordd} and \eqref{oe} that 
\[
\nu_{\Br(i_u)}+\frac{\Br(i_u)-i_u}{2}\not \succ \nu_{i_u}+\nu_{i_u+1} \Longleftrightarrow \nu_{i_u}+\nu_{i_u+1} \gg  \nu_{\Br(i_u)}+\frac{\Br(i_u)-i_u}{2}-1\,\cdot
\]
If $\gamma^{s+1-u}\in \E\cap \Od \subset \Ee$, we then obtain that the partition $s(\delta^{s+1-u}),\gamma^{s+1-u}$ belongs to $\Eee$, so that, by \Lem{lem3} and  \eqref{oe} of \Lem{lem1}, all the crossings in \So of $\Psi$
 are reversible by \St of $\Phi$. We set 
 \[\gamma^{s+1-u} = \pi_1\gg\cdots\gg\pi_r\]
 and if $\nu_{i_u}+\nu_{i_u+1}=s(\delta^{s+1-u})$ crosses all the primary parts up to $\pi_j$, we then have by \eqref{eo} of \Lem{lem1}
 \begin{align*}
\delta^{s+2-u},\gamma^{s+2-u}&=\,\,\delta^{s+1-u}\setminus{\nu_{i_u}+\nu_{i_u+1}}\quad ,\pi_1+1\gg\cdots\gg \pi_j+1\gg\alpha(\nu_{i_u}+\nu_{i_u+1}-j)\\
\mu^{s+2-u}&=\,\,\beta(\nu_{i_u}+\nu_{i_u+1}-j)\succ \pi_{j+1}\succ\cdots\succ \pi_r,\mu_{s+1-u}\,\cdot
\end{align*}
Furthermore, always by condition $(2)$, we have that
\[s(\delta^{s+1-u}\setminus{\nu_{i_u}+\nu_{i_u+1}})=\nu^-(i_u)\gg \nu_{\Br(i_u)}+\frac{\Br(i_u)-i_u}{2}=\pi_1+1\]
so that $\delta^{s+2-u},\gamma^{s+2-u}\in \E$ and  we obtain that $\gamma^{s+2-u}\in \E\cap\Od$ and $s(\delta^{s+2-u})\gg g(\gamma^{s+2-u})$.
\\Moreover, if $\mu_{s+1-u}\in \Od$ and $j<r$, we then have that $(\pi_r,g(\mu_{s+1-u}))$ is the troublesome pair coming from the splitting of 
$\nu_{i_{u+1}}+\nu_{i_{u+1}+1}$ and satisfies $\pi_r\succ g(\mu_{s+1-u})$, so that $\mu^{s+2-u}\in \Od$.
If $\mu_{s+1-u}\in \Od$ and $j=r$, this means that the splitting of $\nu_{i_u}+\nu_{i_u+1}$ occurs in between those of $\nu_{i_{u+1}}+\nu_{i_{u+1}+1}$ and the lower halves are still well-ordered in terms of $\succ$,
so that $\mu^{s+2-u}\in \Od$. In any case, if $\mu_{s+1-u}\in \Od$ (with the previous assumption that $\gamma^{s+1-u}\in \E\cap \Od $), then $\mu^{s+2-u}\in \Od$.\\
\item If we have that $\Br(i_u)=i_u$, then by \Prp{pr5}, the splitting occurs directly and we have 
\[ 
 \nu_{i_u+1}\succ g(\gamma^{s+1-u})\,\cdot
\]
Then we have that  
\begin{align*}
\delta^{s+2-u},\gamma^{s+2-u}&=\,\,\delta^{s+1-u}\setminus{\nu_{i_u}+\nu_{i_u+1}}\quad ,\quad \nu_{i_u}\\
\mu^{s+2-u}&=\nu_{i_u+1},\gamma^{s+1-u},\mu^{s+1-u}\,\cdot
\end{align*}
so that, if $\gamma^{s+1-u}$ and $\mu^{s+1-u}$ are in $\Od$, since $s(\gamma^{s+1-u})\succ g(\mu^{s+1-u})$, we then have that $\mu^{s+2-u}$ is also in $\Od$.
Note that $s(\delta^{s+1-u}\setminus{\nu_{i_u}+\nu_{i_u+1}})=\nu^-(i_u)$.
\begin{itemize}
 \item If $\nu^-(i_u)\triangleright \nu_{i_u}+\nu_{i_u+1}$, then we obtain that
\begin{align*}
\nu^{-}(i_u)-\nu_{i_u}&= \nu^{-}(i_u)-(\nu_{i_u}+\nu_{i_u+1})+\nu_{i_u+1}\\
&\geq 2\quad\quad (\text{by \eqref{Ord} and the fact that}\,\,\nu_{i_u+1}\geq 1)\,,
\end{align*}
so that, by \eqref{cons1} and \eqref{Ord}, $\nu^{-}(i_s)\gg\nu_{i_u}$.
\item In the case that $\nu^{-}(i_s)\nt\nu_{i_u}+\nu_{i_u+1}$, this means by \eqref{Ordd} that we have the case of a pair of secondary parts with colors in $\Ccc$, and which are consecutive for 
$\succ$. Then the pair $(\nu^{-}(i_s),\nu_{i_u}+\nu_{i_u+1})$ has the form $(k_{cd},k_{ab})$ or $((k+1)_{ad},k_{bc})$ for some primary colors $a<b<c<d$. We check the different cases according to the parity of 
$k$ :
\[(2k)_{cd}\gg k_b\,,\quad (2k+1)_{cd}\gg (k+1)_a\,\quad (2k+1)_{ad}\gg k_c\,\quad (2k+2)_{ad}\gg (k+1)_b\,\cdot\]
We then conclude that $\nu^{-}(i_s)\gg\nu_{i_u}$. 
\end{itemize}
In any case, we always have that $\nu_{i_u}$ is well-ordered with the part to its left in terms of $\gg$, 
so that $\delta^{s+2-u},\gamma^{s+2-u} \in \E$, and then $\gamma^{s+2-u}\in \E\cap\Od$ and $s(\delta^{s+2-u})\gg g(\gamma^{s+2-u})$.\\
\end{itemize} 
Note that the process $\Psi$ is reversible by $\Phi$ since the crossings are reversible andso is the splitting. We then obtain \Prp{pr6} recursively on $u$ in decreasing order.
\subsection{Proof of \Prp{pr7}}
Let us take a shortcut $\zeta = \zeta_1+\zeta_2\gg\cdots\gg\zeta_{2s+1}+\zeta_{2s+2}$, and an allowed pattern 
$\eta = \eta_1+\eta_2\gg \cdots \gg \eta_{2t-1}+\eta_{2t}\gg \eta_{2t+1}$ such that $\mathbf{Br}_\eta (1)=2t+1$. Without loss of generality, by adding a constant $k$ to the part $\nu_{2i-1}+\nu_{2i}$, we can suppose that $\zeta_{2s+1}+\zeta_{2s+2}\gg \eta_1+\eta_2$. If we consider the sequence
\[\nu^{(0)}= \zeta_1+\zeta_2\gg\cdots\gg\zeta_{2s+1}+\zeta_{2s+2}\gg\eta_1+\eta_2\gg \cdots \gg \eta_{2t-1}+\eta_{2t}\gg \eta_{2t+1}\,,\]
by adding a large constant $k$ to the parts of the sequence $\nu^{(0)}$, we can say $\eta_{2t+1}$ is the bridge in $\nu$ of all 
$$i\in 2\{0,\ldots,s+t\}+1\,\cdot$$ In fact, by Remark 2.1, we have that the lower halves grow according to $k/2$, so that for some $k$ large enough, $\eta_{2t+1}+k-1$ will be $1$-distant-different from all the lower halves in the sequence $\nu$ in terms of $\succeq$.
We finally consider the sequences of the form
\begin{align*}
\nu^{(u)} &= \zeta_1+\zeta_2+su\gg\cdots\gg\zeta_{2s+1}+\zeta_{2s+2}+su\gg\zeta_1+\zeta_2+s(u-1)\gg\cdots\gg\zeta_{2s+1}+\zeta_{2s+2}+s(u-1)\gg\\
&\qquad \qquad \cdots\gg \zeta_1+\zeta_2+s\gg\cdots\gg\zeta_{2s+1}+\zeta_{2s+2}+s\gg\zeta_1+\zeta_2\gg\cdots\gg\zeta_{2s+1}+\zeta_{2s+2}\gg\\
&  \qquad \qquad\eta_1+\eta_2\gg \cdots \gg \eta_{2t-1}+\eta_{2t}\gg \eta_{2t+1}\,\cdot
\end{align*}
The sequence $\nu$ is well defined, since $\zeta$ is a shorcut, we then have by \eqref{Ord} and \eqref{Ordd} that
\begin{align*}
\overline{\zeta_{2s+1}+\zeta_{2s+2}}\succ \overline{\zeta_1+\zeta_2+1-s}&\Longrightarrow \zeta_{2s+1}+\zeta_{2s+2}\succ \zeta_1+\zeta_2+1-s\\
&\Longrightarrow \zeta_{2s+1}+\zeta_{2s+2}\,\,\triangleright\,\,\zeta_1+\zeta_2-s\\
&\Longrightarrow \zeta_{2s+1}+\zeta_{2s+2}+s\gg\zeta_1+\zeta_2\,,
\end{align*}
so that $\zeta_{2s+1}+\zeta_{2s+2}+su'\gg\zeta_1+\zeta_2+s(u'-1)$ for all $u'\geq 1$. We also have that $\eta_{2t+1}$ is the bridge of all the indices of the secondary parts in $\nu^{(u)}$. In fact, we have by \eqref{aj} that
$$\beta(\zeta_{2s+1}+\zeta_{2s+2}+s)\preceq s+\beta(\zeta_{2s+1}+\zeta_{2s+2})\preceq s+t+\eta_{2t+1} \prec s+t+1+\eta_{2t+1}\,,$$
and we obtain in the same way, that for all $i\in \{0,\ldots,s-1\}$
$$\beta(\zeta_{2i+1}+\zeta_{2i+2}+s)\prec s-i+s+t+1+\eta_{2t+1}\,,$$
so that $\eta_{2t+1}$ is the bridge of all the indices (in the corresponding set $I$) of the parts in $\nu^{(1)}$. Using \eqref{aj} recursively on $u$, we proved that $\eta_{2t+1}$ is indeed the bridge of all indices of the secondary parts in the sequence $\nu^{(u)}$.
\bi To conclude, we see that there are 
$(u+1)(s+1)+t$ secondary parts in $\nu^{(u)}$ (the head included) between $\zeta_1+\zeta_2+su$ and $\eta_{2t+1}$, and we then have 
$$\eta_{2t+1}+ (u+1)(s+1)+t - (\zeta_1+\zeta_2+su) = \eta_{2t+1}- (\zeta_1+\zeta_2)+ t+u+s+1\,\cdot$$
There then exists some $u_0$ such that,
$$\eta_{2t+1}+ (u_0+1)(s+1)+t \succ (\zeta_1+\zeta_2+su_0)\,,$$
so that condition $(2)$ in \Prp{define} is not true. The sequence $\nu^{(u_0)}$ is then a forbidden pattern, and this concludes the proof.
\subsection{Proof of \Prp{pr3}}
Let us take $\nu = (\nu_1,\cdots,\nu_{p+2s})$, with $I=\{i_1<\cdots<i_s\}$ and $J=\{j_1<\cdots<j_p\}$.
\\\\We observe that, in \Prp{pr2},  the sequence $(\delta^v,\gamma^v,\mu^v)$ becomes the sequence $(\delta^{v+1},\gamma^{v+1},\mu^{v+1})$ after applying 
\So once to the secondary part $s(\delta^v)$, and some iterations of \St by crossing the secondary part with some primary parts of $\gamma^v$.
This means that once we obtain the sequence $\mu^v$, it is no longer affected by the process $\Psi$.
\begin{itemize}
 \item Since we never cross two primary parts in the process, once we have the splitting $s(\gamma^{v}),g(\mu^{v})$, their relative position in the remainer of the process $\Psi$
 is unchanged. We then obtain that the upper and the lower halves' positions satisfy $\theta_{i_{s-v+2}}<\theta_{i_{s-v+2}+1}$.\\
 \item The passage from the secondary part $s(\delta^v)$ to its splitting to become $s(\gamma^{v+1}),g(\mu^{v+1})$ implies that
 the position of the lower part increases during the crossings, and then is fixed after the splitting.
 We thus obtain that $\theta_{i_{s+1-v}+1}$ is the position of the $g(\mu^{v+1})$. With the fact that the sequence $g(\mu^v)$ is the strict tail of $g(\mu^{v+1})$, we reach the inequality
 $\theta_{i_{s-v+2}+1}>\theta_{i_{s+1-v}+1}\geq i_{s+1-v}+1$.
 This gives the first inequality of \eqref{i1}.\\
 \item  If the splitting of $s(\delta^v)$ occurs before $g(\gamma^v)$, it means that $g(\gamma^v)$ belongs to $\mu^{v+1}$, and the position of the corresponding upper half is fixed in the rest of the process.
We then have that $\theta_{i_{s-v+2}}>\theta_{i_{s+1-v}+1}\,\cdot$
 Otherwise, the splitting of $s(\delta^v)$ occurs between $g(\gamma^v)$ and $g(\mu^v)$, and the relative position of the corresponding upper halves will not change until the end of the process.
 We thus have that $\theta_{i_{s+1-v}+1}>\theta_{i_{s+1-v}}>\theta_{i_{s-v+2}}$,
 and this leads (recursively on $v$) to the proof of \eqref{ii}.\\
 \item Recall that we never cross two primary parts in the process, and this naturally leads to
 $\theta_{j_v}<\theta_{j_{v+1}}$, 
 for $j_v<j_{v+1}$ and we have \eqref{jj}. Moreover, the primary parts can only move backward, since they can only cross some secondary parts to their left. We then obtain the second inequality of \eqref{i1} 
 $\theta_{j_v}\leq j_v$.\\
. \item Since the crossing only occurs between the secondary and primary parts, if the secondary part corresponding to $i$ does not cross in the primary part corresponding to $j$, then 
 we have that $\theta_{i+1}<\theta_j\,,$
 and if they crossed, then both the upper and the lower halves move together, and in the remainder of the process, their relative positions stay unchanged, so that 
 $\theta_j<\theta_{i}$, and we obtain \eqref{ij}.
\end{itemize}
 \subsection{Proof of \Prp{pr4}}
 We saw in the previous proof that, since the positions of the lower halves are increasing, 
for any $i_u\in I$, the crossings can occur with primary parts coming from some indices $J$ or in $I$. We then look for $x\in J\cup I$ such that $x>i_u$ and $\theta_x<\theta_{i_u}$.
Let us then set $\{x_1,\ldots, x_v\}=\{x\in J\cup I: x>i_u,\,,\theta_{x}<\theta_{i_u}\}$ such that
$$
 \theta_{x_1}<\cdots<\theta_{x_v}\,\cdot
$$
Note that if $\{x\in J\cup I: x>i_u,\,,\theta_{x}<\theta_{i_u}\} = \emptyset$, then the splitting occurs directly and \[\Br(i_u)=i_u=\max_{x\in I}\{x\geq i_u , \theta_x\leq \theta_{i_u}\}\,\cdot\]
 Recall that if $\{x\in J\cup I: x>i_u,\,,\theta_{x}<\theta_{i_u}\}\neq \emptyset$, we then have 
$$
  \theta_{x_v}<\theta_{i_u}<\theta_{i_u+1}\quad\text{and}\quad x_1,\ldots, x_v>i_u\,\cdot
$$
 \begin{itemize}
  \item If $\{x_1,\ldots, x_v\}\cap J \neq \emptyset$, then we necessarily have that $x_1\in J$. In fact, suppose that $x_1\in I$ and $x_1<x \in \{x_1,\ldots, x_v\}\cap J$.
  Since $x_1>i_u$, by \eqref{ii}, we have $\theta_{i_u+1}<\theta_{x_1+1}$
  and then 
  \[\theta_{x_1}<\theta_x<\theta_{i_u}<\theta_{i_u+1}<\theta_{x_1+1}\,,\]
  and this contradicts \eqref{ij}. Furthermore, by \eqref{jj}, we have that 
$$
   x_1 = \min \{x_1,\ldots, x_v\}\cap J= \min_{x\in J}\{ x>i_u , \theta_x<\theta_{i_u}\}\,\cdot
 $$
  \item Otherwise, we have $\{x_1,\ldots, x_v\}\cap J =\emptyset$.  In that case, $\{x_1,\ldots, x_v\}\subset I$.
  We then have that 
  $x_1>\cdots>x_v$. In fact, for any $x<x'\in \{x_1,\ldots, x_v\}$, by \eqref{ii}, we have 
  \[\theta_{i_u+1}<\theta_{x+1}<\theta_{x'+1}\,,\]
  and if we suppose that $\theta_{x}<\theta_{x'}$,
  we then obtain the inequality 
  \[\theta_{x}<\theta_{x'}<\theta_{i_u}<\theta_{i_u+1}<\theta_{x+1}<\theta_{x'+1}\,,\]
  and this contradicts \eqref{ii}.
  Furthermore, this leads to the following relation
  $$
    x_1 = \max \{x_1,\ldots, x_v\}=\max_{x\in I}\{x\geq i_u , \theta_x\leq \theta_{i_u}\}\,\cdot
 $$
  \end{itemize}
In any case, by \Prp{pr5}, we have that $x_1=\Br(i)$. In fact, $x_1$ is the index of the first crossed part.


\begin{thebibliography}{999}

\bibitem{AAB03}
K. ALLADI, G.E. ANDREWS and A. BERKOVICH, \emph{A new four parameter $q$-series identity and its partition implications}, Invent. Math. \textbf{153} (2003), 231--260.

\bibitem{AAG95}
K. ALLADI, G.E. ANDREWS and B. GORDON, \emph{Generalizations and refinements of a partition
theorem of G\"ollnitz}, J. Reine Angew. Math. \textbf{460} (1995), 165--188.

\bibitem{AG93}
K. ALLADI and B. GORDON, \emph{Generalization of Schur's partition theorem}, Manuscripta Math. \textbf{79} (1993), 113--126.

\bibitem{AN68}
G.E. ANDREWS, \emph{A new generalization of Schur's second partition theorem}, Acta Arith. \textbf{14} (1968)
, 429--434.

\bibitem{AN69}
G.E.  ANDREWS, \emph{On  a  partition  theorem  of  G\"ollnitz  and  related  formula}, J.  Reine Angew. Math. \textbf{236} (1969), 37--42.

\bibitem{BR80}
D. BRESSOUD, \emph{A combinatorial proof of Schur's 1926 partition theorem}, Proc. Amer. Math. Soc. \textbf{79} (1980), 338--340.

\bibitem{CL06}
S. CORTEEL and J. LOVEJOY, \emph{An iterative-bijective approach to generalizations of Schur's theorem}, Eur. J. Comb. \textbf{27} (2006), 496--512.

\bibitem{DS77}
R. DONAGHEY, L. W. SHAPIRO, \emph{Motzkin numbers}, J. Combin. Theory Ser. A  \textbf{23 (3)} (1977), 291--301.

\bibitem{IK19}
I. KONAN, \emph{A Bijective proof and Generalization of Siladi\'c's theorem}, arXiv:1906.00089, submitted.

\bibitem{IK}
I. KONAN, \emph{Beyond G\"ollnitz' theorem I: a bijective approach}, arXiv:1909.00364.


\bibitem{GO67}
H. G\"OLLNITZ, \emph{Partitionen mit Differenzenbedingungen}, J. Reine Angew. Math. \textbf{225} (1967),
154--190.
 

\bibitem{PRS04}
PADMAVATHAMMA, M. RUDY SALESTINA and S.R. SUDARSHAN, \emph{Combinatorial proof of the G\"ollnitz's theorem on partitions}, Adv. Stud. Contemp. Math. \textbf{8} (2004), no.1,  47--54.

\bibitem{RS99}
R. STANLEY, \emph{Enumerative Combinatorics}, doi 10.1007/978-1-4615-9763-6, (1999).

\bibitem{Sc26}
I. SCHUR, \emph{Zur additiven zahlentheorie}, Sitzungsberichte der Preussischen Akademie der Wissenschaften (1926) , 488--495.

\bibitem{ZJ15}
J.Y.J ZHAO, \emph{ A bijective proof of the Alladi-Andrews-Gordon
partition theorem}, Electron. J. Combin. \textbf{22} (2015), no. 1, Paper 1.68

\end{thebibliography}
\end{document}